\documentclass[final,reqno]{amsart}

\usepackage{amsthm}

\newtheorem{theorem}{Theorem}[section]
\newtheorem{lemma}[theorem]{Lemma}

\newtheorem{corollary}[theorem]{Corollary}

\newtheorem{definition}[theorem]{Definition}

\newtheorem{proposition}[theorem]{Proposition}

\theoremstyle{definition}

\theoremstyle{remark}
\newtheorem{remark}[theorem]{Remark}

\usepackage[utf8]{inputenc}
\usepackage{url}
\usepackage{graphics}
\usepackage{amsmath}
\usepackage{amssymb}
\usepackage{mathrsfs}
\usepackage[ruled,commentsnumbered,norelsize]{algorithm2e}
\usepackage{cite}
\usepackage[usenames,dvipsnames]{color}
\usepackage{enumerate}
\usepackage{comment}
\usepackage{subfig}
\usepackage{multirow}
\usepackage[english]{babel}
\usepackage[colorlinks,
            bookmarksnumbered,
            citecolor=red,
            urlcolor=black,
            linkcolor=blue
            ]{hyperref}
\usepackage{fancyhdr}

\usepackage[stable]{footmisc}
\usepackage{marvosym} 

\usepackage{ifpdf}
\usepackage{color}
\ifpdf
\usepackage[pdftex]{graphicx}
\else
\usepackage{graphicx}
\fi

\newcommand{\NN}{\mathcal N}

\newcommand{\Ss}{\mathbb S}

\newcommand{\eqdef}{\triangleq}
\newcommand{\Rr}{\mathbb R}

\newcommand{\Nn}{\mathbb N}

\newcommand{\cl}{\mathop{\rm cl}\nolimits}
\newcommand{\dom}{\mathop{\rm dom}\nolimits}

\newcommand{\cone}{\mathop{\rm cone}\nolimits}
\newcommand{\conv}{\mathop{\rm conv}\nolimits}

\newcommand{\aff}{\mathop{\rm aff}\nolimits}
\newcommand{\Int}{\mathop{\rm int}\nolimits}

\newcommand{\diag}{\mathop{\rm diag}\nolimits}

\newcommand{\prox}{\mathop{\rm prox}\nolimits}

\newcommand{\sign}{\mathop{\rm sign}\nolimits}

\newcommand{\minimize}{\operatorname{minimize}}  
\renewcommand{\Re}{{\rm{I\!R}} }
\renewcommand{\Nn}{{\rm{I\!N}} }

\DeclareMathOperator*{\argmin}{\arg\!\min}
\renewcommand{\vec}{\mathop{\rm vec}\nolimits}
\newcommand{\ie}{\emph{i.e.},~}
\newcommand{\matlab}{{\sc{Matlab}}}
\newcommand{\gurobi}{{\sc{Gurobi}}}

\numberwithin{equation}{section}

\usepackage{etoolbox}
\newtoggle{svver}

\togglefalse{svver}

\begin{document}

\title[FB truncated Newton methods for convex composite optimization]
{Forward-backward truncated Newton methods for convex composite optimization\footnotesize{\textsuperscript{\textnormal{1}}}}

\author[P. Patrinos]{Panagiotis Patrinos}
\address[P. Patrinos]{IMT Institute for Advanced Studies Lucca}
\email{panagiotis.patrinos@imtlucca.it}
\author[L. Stella]{Lorenzo Stella}
\address[L. Stella]{IMT Institute for Advanced Studies Lucca}
\email{lorenzo.stella@imtlucca.it}
\author[A. Bemporad]{Alberto Bemporad}
\address[A. Bemporad]{IMT Institute for Advanced Studies Lucca}
\email{alberto.bemporad@imtlucca.it}

\begin{abstract}
This paper proposes two proximal Newton-CG methods for convex nonsmooth optimization problems in composite form.
The algorithms are based on a a reformulation of the original nonsmooth problem as
the unconstrained minimization of a continuously differentiable function, namely the \emph{forward-backward envelope (FBE)}.
The first algorithm is based on a standard line search strategy, whereas the second one 
combines the global efficiency estimates of the corresponding first-order methods, while achieving fast asymptotic convergence rates. Furthermore, they are computationally attractive since each Newton iteration requires the approximate solution of a linear system of usually small dimension. 
\end{abstract}


\maketitle

\footnotetext[1]{A preliminary version of this paper \cite{patrinos2013proximal}
was presented at the 52nd IEEE Conference on Decision and Control, Florence,
Italy, December 11, 2013.}

\setcounter{footnote}{1}

\section{Introduction}
The focus of this work is on efficient Newton-like algorithms for convex
optimization problems in composite form, \ie
\begin{equation}\label{eq:GenProb}
\minimize\ F(x) = f(x)+g(x),
\end{equation}
where $f\in\mathcal{S}_{\mu_f,L_f}^{2,1}(\Re^n)$\footnote{$\mathcal{S}_{\mu,L}^{2,1}(\Re^n)$: class of twice continuously differentiable,
strongly convex functions with modulus of strong convexity $\mu\geq 0$, whose gradient is Lipschitz continuous with constant $L\geq 0$.} and
$g\in\mathcal{S}^0(\Re^n)$\footnote{$\mathcal{S}^0(\Re^n)$: class of proper, lower semicontinuous, convex functions from $\Re^n$ to $\overline{\Re} = \Re\cup\{+\infty\}$.}
has a cheaply computable proximal mapping~\cite{moreau1965proximiteet}.
Problems of the form~\eqref{eq:GenProb} are abundant in many scientific areas
such as control, signal processing, system identification, machine learning and
image analysis, to name a few. For example, when $g$ is the indicator of a
convex set then~\eqref{eq:GenProb} becomes a constrained optimization problem,
while for $f(x)=\|Ax-b\|_2^2$ and $g(x)=\lambda\|x\|_1$ it becomes the
$\ell_1$-regularized least-squares problem which is the main building block of
compressed sensing. When $g$ is equal to the nuclear norm, then
problem~\eqref{eq:GenProb} can model low-rank matrix recovery problems.
Finally, conic optimization problems such as LPs, SOCPs and SPDs can be brought
into the form of~\eqref{eq:GenProb}, see \cite{lan2011primal}.

Perhaps the most well known algorithm for problems in the form \eqref{eq:GenProb}
is the forward-backward splitting (FBS) or proximal gradient method
\cite{lions1979splitting, combettes2011proximal}, a generalization of the
classical gradient and gradient projection methods to problems involving a
nonsmooth term. Accelerated versions of FBS, based on the work of Nesterov
\cite{nesterov2007gradient,beck2009fast,tseng2008accelerated}, have also gained
popularity.
Although these algorithms share favorable global 
convergence rate estimates of order $O(\epsilon^{-1})$ or $O(\epsilon^{-1/2})$
(where $\epsilon$ is the solution accuracy), they are first-order methods and
therefore usually effective at computing solutions of low or medium accuracy only.
An evident remedy is to include second-order information by replacing the
Euclidean norm in the proximal mapping with the $Q$-norm,
where $Q$ is the Hessian of $f$ at $x$ or some approximation of it, mimicking
Newton or quasi-Newton methods for unconstrained problems.
This route is followed in the recent work of \cite{becker2012quasi, Lee2012ProximalNIPS}.
However, a severe limitation of the approach is that, unless $Q$ has a special
structure, the linearized subproblem is very hard to solve. For example, if
$F$ models a  QP, the corresponding subproblem is as hard as the original problem.

In this paper we follow a different approach by defining a function, which
we call \emph{forward-backward envelope (FBE)}, that has favorable properties and
can serve as a real-valued, smooth, exact penalty function
for the original problem. Our approach combines and extends ideas stemming from
the literature on merit functions for \emph{variational inequalities}
(VIs) and \emph{complementarity problems} (CPs), specifically the reformulation of a VI as a constrained continuously differentiable optimization problem
via the regularized gap function \cite{fukushima1992equivalent} and as an unconstrained continuously differentiable optimization problem via the
D-gap function \cite{yamashita1997unconstrained} (see \cite[Ch. 10]{facchinei2003finite} for a survey and \cite{Li2007exact}, \cite{patrinos2011global}
for applications to constrained optimization and model predictive control of dynamical systems).

Next, we show that one can design Newton-like methods to minimize the FBE by using tools from nonsmooth analysis. Unlike the approaches of
\cite{becker2012quasi, Lee2012ProximalNIPS}, where the corresponding subproblems are expensive to solve, our algorithms require only the solution
of a  usually small linear system  to compute the Newton direction. However, this work focuses on devising algorithms that have good {complexity guarantees}
provided by a global (non-asymptotic) convergence rate while achieving $Q$-superlinear or $Q$-quadratic\footnote{A sequence $\{x^k\}_{k\in\Nn}$ converges
to $x_\star$ with $Q$-superlinear rate if $\frac{\|x^{k+1}-x_\star\|}{\|x_k-x_\star\|}\to 0$. It converges to $x_\star$ with $Q$-quadratic rate if there
exists a $\bar{k}>0$ such that $\frac{\|x^{k+1}-x_\star\|}{\|x_k-x_\star\|^2}\leq M$, for some $M>0$ and all $k\geq\bar{k}$.}
asymptotic convergence rates in the nondegenerate cases. We show that one can achieve this goal by interleaving Newton-like iterations on the FBE
and FBS iterations. This is possible by relating directions of descent for the considered penalty function with those for the original nonsmooth function. 

The main contributions of the paper can be summarized as follows. We show
how Problem~\eqref{eq:GenProb} can be reformulated as the unconstrained minimization
of a real-valued, continuously differentiable function, the FBE, 
providing a framework that allows to extend classical algorithms for smooth
unconstrained optimization to nonsmooth or constrained problems in
composite form~\eqref{eq:GenProb}. Moreover, based on this framework, we
devise efficient proximal Newton algorithms with $Q$-superlinear or $Q$-quadratic asymptotic
convergence rate to solve~\eqref{eq:GenProb}, with global complexity bounds.
The conjugate gradient (CG) method is employed to compute efficiently 
an approximate Newton direction at every iteration. Therefore our algorithms are
able to handle large-scale problems since they require only the calculation of matrix-vector
products and there is no need to form explicitly the generalized Hessian matrix.

The outline of the paper is as follows. In Section~\ref{sec:FBE} we introduce the
FBE, a continuously differentiable penalty function for
\eqref{eq:GenProb}, and discuss some of its properties. In Section~\ref{sec:LNA} we
discuss the generalized differentiability properties of the gradient of the FBE and introduce
a linear Newton approximation (LNA) for it, which plays a role similar to that of the Hessian
in the classical Newton method. Section~\ref{sec:FBNCG} is the core of the
paper, presenting two algorithms for solving Problem~\eqref{eq:GenProb} and discussing their
local and global convergence properties. In Section~\ref{sec:Examples} we
consider some examples of $g$ and discuss the generalized Jacobian of
their proximal operator, on which the LNA
is based. Finally, in Section~\ref{sec:Simulations}, we consider some practical
problems and show how the proposed methods perform in solving them.

\section{Forward-backward envelope}\label{sec:FBE}
In the following we indicate by $X_\star$ and $F_\star$, respectively, the
set of solutions of problem~\eqref{eq:GenProb} and its optimal objective value.
Forward-backward splitting for solving~\eqref{eq:GenProb} relies on computing,
at every iteration, the following update
\begin{equation}\label{eq:FBS}
x^{k+1} = \prox_{\gamma g}(x^k-\gamma\nabla f(x^k)),
\end{equation}
where the \emph{proximal mapping}~\cite{moreau1965proximiteet} of $g$
is defined by 
\begin{equation}\label{eq:prox}
\prox_{\gamma g}(x)  \eqdef\argmin_u\left\{g(u)+\tfrac{1}{2\gamma}\|u-x\|^2\right\}.
\end{equation}
The value function  of the optimization problem~\eqref{eq:prox}  defining the proximal mapping
is called the \emph{Moreau envelope} and is denoted by $g^\gamma$, \ie
\begin{equation}\label{eq:MoreauEnv}
 g^{\gamma}(x)  \eqdef\inf_u\left\{g(u)+\tfrac{1}{2\gamma}\|u-x\|^2\right\}.
\end{equation}
Properties of the Moreau envelope and the proximal mapping are well documented in the literature \cite{bauschke2011convex,rockafellar2011variational,combettes2005signal,combettes2011proximal}. 
For example, the proximal mapping is single-valued, continuous and nonexpansive (Lipschitz continuous with Lipschitz $1$) 
and the envelope function $g^{\gamma}$ is convex, continuously differentiable, with $\gamma^{-1}$-Lipschitz continuous gradient
\begin{equation}\label{eq:nabla_e}
\nabla g^{\gamma}(x)=\gamma^{-1}(x-\prox_{\gamma g}(x)).
\end{equation}
We will next proceed to the reformulation of~\eqref{eq:GenProb} as the minimization of an unconstrained continuously differentiable function.
It is well known \cite{bauschke2011convex} that an optimality condition
for~\eqref{eq:GenProb} is
\begin{equation}\label{eq:OptCond}
x=\prox_{\gamma g}(x-\gamma\nabla f(x)).
\end{equation}
Since $f\in\mathcal{S}_{\mu_f,L_f}^{2,1}(\Re^n)$, we have that $\|\nabla^2 f(x)\|\leq L_f$ \cite[Lem. 1.2.2]{nesterov2003introductory}, therefore $I-\gamma\nabla^2 f(x)$ is symmetric
and positive definite whenever $\gamma\in(0,1/L_f)$.
Premultiplying both sides of~\eqref{eq:OptCond} by $\gamma^{-1}(I-\gamma\nabla^2 f(x))$, $\gamma\in(0,1/L_f)$,
one obtains the equivalent condition
\begin{equation}\label{eq:OptCondScaled}
\gamma^{-1}(I-\gamma\nabla^2f(x))(x-\prox_{\gamma g}(x-\gamma\nabla f(x))) = 0.
\end{equation}
The left-hand side of equation~\eqref{eq:OptCondScaled} is the gradient of the function that we call \emph{forward-backward envelope},
indicated by $F_\gamma$. Using~\eqref{eq:nabla_e} to integrate~\eqref{eq:OptCondScaled}, one obtains
the following definition.
\begin{definition}
Let $F(x) = f(x)+g(x)$, where $f\in\mathcal{S}_{\mu_f,L_f}^{2,1}(\Re^n)$,
$g\in\mathcal{S}^0(\Re^n)$. The forward-backward envelope of $F$ is given by
\begin{equation}\label{eq:Penalty}
F_\gamma(x)\eqdef f(x)-\tfrac{\gamma}{2}||\nabla f(x)||_2^2+g^{\gamma}(x-\gamma\nabla f(x)).
\end{equation}
\end{definition}
Alternatively, one can express $F_\gamma$ as the value function of the minimization
problem that yields forward-backward splitting. In fact
\begin{subequations}
\begin{align}
F_\gamma(x)&=\min_{u\in\Re^n}\left\{f(x)+\nabla f(x)'(u-x)+g(u)+\tfrac{1}{2\gamma}\|u-x\|^2\right\}\label{eq:Fmin}\\
&=f(x)+g(P_{\gamma}(x))-\gamma\nabla f(x)'G_{\gamma}(x)+\tfrac{\gamma}{2}\|G_{\gamma}(x)\|^2,
\end{align}
\end{subequations}
where
\begin{align*}
P_{\gamma}(x)&\eqdef\prox_{\gamma g}(x-\gamma\nabla f(x)),\\
G_{\gamma}(x)&\eqdef\gamma^{-1}(x-P_{\gamma}(x)).
\end{align*}
One distinctive feature of $F_\gamma$ is the fact that it is real-valued despite the fact that $F$ can be extended-real-valued.
In addition,  $F_\gamma$ enjoys favorable properties, summarized in the next theorem. 
\begin{theorem}\label{Th:PropFg}
The following properties of $F_\gamma$ hold:
\begin{enumerate}[\rm (i)]
\item\label{prop:DerPen} $F_\gamma$ is continuously differentiable with
\begin{equation}\label{eq:DerPen}
\nabla F_\gamma(x)=\left(I-\gamma\nabla^2 f(x)\right)G_{\gamma}(x).
\end{equation}
If $\gamma\in (0,1/L_f)$ then the set of stationary points of $F_\gamma$ equals $X_\star$.
\item\label{prop:UppBnd} For any $x\in\Re^n$, $\gamma>0$
\begin{equation}\label{eq:UppBnd}
F_\gamma(x)\leq F(x)-\tfrac{\gamma}{2}\|G_{\gamma}(x)\|^2.
\end{equation}
\item\label{prop:LowBnd} For any $x\in\Re^n$, $\gamma>0$
\begin{equation}\label{eq:LowBnd}
F(P_{\gamma}(x))\leq F_\gamma(x)-\tfrac{\gamma}{2} \left(1-{\gamma}L_f\right)\|G_{\gamma}(x)\|^2.
 \end{equation}
In particular, if $\gamma\in\left(0,1/L_f\right]$ then
\begin{equation}\label{eq:LowBnd4Gamma}
F(P_{\gamma}(x))\leq F_\gamma(x).
\end{equation}
\item\label{cor:Equivalence} If $\gamma\in (0,1/L_f)$ then $X_\star=\argmin F_\gamma$.
\end{enumerate}
\end{theorem}
\begin{proof}
Part (i) has already been proven. 
Regarding (ii), from the optimality condition for the problem defining the proximal mapping we have
\[
G_{\gamma}(x)-\nabla f(x)\in\partial g(P_{\gamma}(x)),
\]
\ie $G_{\gamma}(x)-\nabla f(x)$ is a subgradient of $g$ at $P_{\gamma}(x)$. From the subgradient inequality
\begin{align*}
g(x)&\geq g(P_{\gamma}(x))+(G_{\gamma}(x)-\nabla f(x))'(x-P_{\gamma}(x))\\
&=g(P_{\gamma}(x))-\gamma\nabla f(x)'G_{\gamma}(x)+\gamma\|G_{\gamma}(x)\|^2
\end{align*}
Adding $f(x)$ to both sides proves the claim.
For part (iii), we have
\begin{align*}
F_\gamma (x)&=f(x)+\nabla f(x)'(P_{\gamma}(x)-x)+g(P_{\gamma}(x)){+}\tfrac{\gamma}{2}\|G_{\gamma}(x)\|^2\\
&\geq f(P_{\gamma}(x))+g(P_{\gamma}(x))-\tfrac{L_f}{2}\|P_{\gamma}(x)-x\|^2+\tfrac{\gamma}{2}\|G_{\gamma}(x)\|^2.
\end{align*}
where the inequality follows by Lipschitz continuity of $\nabla f$ and the descent lemma, see e.g.~\cite[Prop. A.24]{bertsekas1999nonlinear}. For part (iv), putting $x_\star\in X_\star$ in~\eqref{eq:UppBnd} and~\eqref{eq:LowBnd} and using $x_\star=P_\gamma(x_\star)$ we obtain $F(x_\star)=F_\gamma(x_\star)$. Now, for any $x\in\Re^n$ we have $F_\gamma(x_\star)=F(x_\star)\leq F(P_\gamma(x))\leq F_\gamma(x)$, where the first inequality follows by optimality of $x_\star$ for $F$, while the second inequality follows  by~\eqref{eq:LowBnd}. This shows that every $x_\star\in X_\star$ is also a (global) minimizer of $F_\gamma$. The proof finishes by recalling that the set of minimizers of $F_\gamma$ are a subset of the set of its stationary points, which by (i) is equal to $X_\star$.
\iftoggle{svver}{\qed}{}
\end{proof}
Parts (i) and (iv) of Theorem~\eqref{Th:PropFg} show that if $\gamma\in (0,1/L_f)$, the nonsmooth problem~\eqref{eq:GenProb} is completely equivalent to the unconstrained minimization of the continuously differentiable function $F_\gamma$, in the sense that the sets of minimizers and optimal values are equal. In other words we have
$$\argmin F=\argmin F_\gamma,\qquad \inf F = \inf F_\gamma.$$
Part (ii) shows that an $\epsilon$-optimal solution $x$ of $F$ is automatically $\epsilon$-optimal for $F_\gamma$, while part (iii) implies that from an $\epsilon$-optimal for $F_\gamma$ we can directly obtain an $\epsilon$-optimal solution for $F$ if $\gamma$ is chosen sufficiently small, \ie
\begin{align*}
F(x)-F_\star&\leq\epsilon\implies F_\gamma(x)-F_\star\leq\epsilon,\\
F_\gamma(x)-F_\star&\leq\epsilon\implies F(P_\gamma(x))-F_\star\leq\epsilon.
\end{align*}

Notice that part (iv) of Theorem~\ref{Th:PropFg} states that if $\gamma\in (0,1/L_f)$, then not only do the stationary points of $F_\gamma$ agree with $X_\star$ (cf. Theorem~\ref{Th:PropFg}(\ref{prop:DerPen})), but also that its set of minimizers agrees with $X_\star$, \ie although $F_\gamma$ may not be convex, the set of stationary points turns out to be equal to the set of its minimizers. However, in the particular but important case where $f$ is convex quadratic, the FBE is convex with Lipschitz continuous gradient, as the following theorem shows.
\begin{theorem}\label{th:ProxPropQuad}
If $f(x)=\tfrac{1}{2}x'Qx+q'x$ and $\gamma\in(0,1/L_f)$, then $F_\gamma\in\mathcal{S}^{1,1}_{\mu_{F_\gamma},L_{F_\gamma}}(\Re^n)$, where
\begin{subequations}
\begin{align}
L_{F_\gamma}&=2(1-\gamma\mu_f)/\gamma,\label{eq:LipF}\\
\mu_{F_\gamma}&=\min\{(1-\gamma\mu_f)\mu_f,(1-\gamma L_f)L_f\}\label{eq:muF}
\end{align}
\end{subequations}
and $\mu_f=\lambda_{\min}(Q)\geq 0$, $L_f=\lambda_{\max}(Q)$.  
\end{theorem}
\begin{proof}
Let
\begin{align*}
\psi_1(x) & \eqdef f(x)-(\gamma/2)\|\nabla f(x)\|^2=(1/2)x'Q(I-\gamma Q)x-\gamma q'Qx-\gamma q'q,\\
\psi_2(x) & \eqdef g^\gamma(x-\gamma\nabla f(x)).
\end{align*}
Due to Lemma~\ref{lem:eigen} (in the Appendix), $\psi_1$ is strongly convex with modulus $\mu_{F_\gamma}$. 
Function $\psi_2(x)$ is convex, as the composition of the convex function $g^\gamma$
with the linear mapping $x-\gamma\nabla f(x)$.
Therefore, $F_\gamma(x)=\psi_1(x)+\psi_2(x)$ is strongly convex with convexity
parameter $\mu_{F_\gamma}$.
On the other hand, for every $x_1,x_2\in\Re^n$
\begin{align*}
\|\nabla F_\gamma(x_1)-\nabla F_\gamma(x_2)\| &\leq\|I-\gamma Q\|\|G_{\gamma}(x_1)-G_{\gamma}(x_2)\|\\
&\leq 2(1-\gamma\mu_f)/\gamma\|x_1-x_2\|
\end{align*}\
where the second inequality is due to Lemma~\ref{le:zNonExp} in the Appendix.
\iftoggle{svver}{\qed}{}
\end{proof}
Notice that if $\mu_f>0$ and we choose $\gamma=1/(L_f+\mu_f)$, then $L_{F_\gamma}=2L_f$ and $\mu_{F_\gamma}=L_f\mu_f/(L_f+\mu_f)$, so $L_{F_\gamma}/\mu_{F_\gamma}=2(L_f/\mu_f+1)$. In other words the condition number of $F_\gamma$ is roughly double compared to that of $f$.
\subsection{Interpretations}
It is apparent from~\eqref{eq:FBS} and~\eqref{eq:OptCond} that FBS is a Picard
iteration for computing a fixed point of the nonexpansive mapping $P_\gamma$.
It is well known that fixed-point iterations may exhibit slow asymptotic
convergence. On the other hand, Newton methods achieve much faster asymptotic
convergence rates. However, in order to devise globally convergent Newton-like
methods one needs a merit function on which to perform a line search, in order
to determine a step size that guarantees sufficient decrease and damps the
Newton steps when far from the solution. This is exactly the role that FBE plays
in this paper. 

Another interesting observation is that the FBE provides a link between gradient methods and FBS, just like 
the Moreau envelope~\eqref{eq:MoreauEnv} does for the proximal point algorithm~\cite{rockafellar1976monotone}.
To see this, consider the problem
\begin{equation}\label{eq:NSprob}
\minimize\ g(x)
\end{equation}
where $g\in\mathcal{S}^0(\Re^n)$. The proximal point algorithm for
solving~\eqref{eq:NSprob} is 
\begin{equation}\label{eq:ProxMin}
x^{k+1}=\prox_{\gamma g}(x^k).
\end{equation}
It is well known that the proximal point algorithm can be interpreted as
a gradient method for minimizing the Moreau envelope of $g$,
cf.~\eqref{eq:MoreauEnv}. Indeed, due to~\eqref{eq:nabla_e},
iteration~\eqref{eq:ProxMin} can be expressed as
$$x^{k+1}=x^k-\gamma\nabla g^\gamma(x^k).$$
This simple idea provides a link between nonsmooth and smooth optimization and
has led to the discovery of a variety of algorithms for problem~\eqref{eq:NSprob},
such as semismooth Newton methods~\cite{fukushima1996globally},
variable-metric~\cite{bonnans1995family} and quasi-Newton methods~\cite{mifflin1998quasi},
and trust-region methods~\cite{sagara2005trust}, to name a few.  
However, when dealing with composite problems, even if $\prox_{\gamma g}$ and $g^\gamma$
are cheaply computable, computing proximal mapping and Moreau envelope of
$(f+g)$ is usually as hard as solving~\eqref{eq:GenProb} itself.
On the other hand, forward-backward splitting takes advantage of the
structure of the problem 
by operating separately on the two summands, cf.~\eqref{eq:FBS}.
The question that naturally arises is the following: 
\begin{quote}
\emph{Is there a continuously differentiable function that provides an
interpretation of FBS as a gradient method, just like the Moreau envelope does
for the proximal point algorithm and problem~\eqref{eq:NSprob}?}
\end{quote}
The forward-backward envelope provides an affirmative answer. Specifically, FBS
can be interpreted as the following  (variable metric) gradient method on the FBE:
$$x^{k+1}=x^k-\gamma(I-\gamma\nabla^2 f(x^k))^{-1}\nabla F_\gamma(x^k).$$
Furthermore, the following properties holding for $g^\gamma$
\begin{equation*}
g^{\gamma} \leq g,\quad\inf g^{\gamma} = \inf g,\quad\argmin g^{\gamma} = \argmin g.
\end{equation*}
correspond to Theorem~\ref{Th:PropFg}(\ref{prop:LowBnd}) and Theorem~\ref{Th:PropFg}(\ref{cor:Equivalence})
for the FBE.
The relationship between Moreau envelope and forward-backward envelope is then apparent.
This opens the possibility of extending FBS and devising new
algorithms for problem~\eqref{eq:GenProb} by simply reconsidering 
and appropriately adjusting methods for unconstrained minimization of
continuously differentiable functions, the most well studied problem in
optimization. In this work we exploit one of the numerous alternatives, by
devising Newton-like algorithms that are able to achieve fast asymptotic
convergence rates.
The next section deals with the other obstacle that needs to be overcome,
\ie constructing a second-order expansion for the $\mathcal{C}^1$
(but not $\mathcal{C}^2$) function $F_\gamma$ around any optimal solution,
that behaves similarly to the Hessian for $\mathcal{C}^2$ functions and allows
us to devise algorithms with fast local convergence.

\section{Second-order Analysis of $F_\gamma$}\label{sec:LNA}
As it was shown in Section~\ref{sec:FBE}, $F_\gamma$ is continuously differentiable over $\Re^n$. However $F_\gamma$ fails to be $\mathcal{C}^2$ in most cases:
since $g$ is nonsmooth, its Moreau envelope $g^\gamma$ is hardly ever $\mathcal{C}^2$. For example, if $g$ is real-valued then $g^\gamma$
is  $\mathcal{C}^2$ and $\prox_{\gamma g}$ is $\mathcal{C}^1$ if and only if $g$ is $\mathcal{C}^2$~\cite{lemarechal1997practical}.
Therefore, we hardly ever have the luxury of assuming continuous differentiability of $\nabla F_\gamma$ and we must resort into generalized notions of
differentiability stemming from nonsmooth analysis. Specifically, our analysis is largely based upon  generalized differentiability properties
of $\prox_{\gamma g}$ which we study next.

\subsection{Generalized Jacobians of proximal mappings}
Since $\prox_{\gamma g}$ is globally Lipschitz continuous, by Rademacher's theorem
\cite[Th.~9.60]{rockafellar2011variational} it is almost everywhere differentiable.  
Recall that Rademacher's theorem asserts that if a mapping $G:\Re^n\to\Re^m$ is locally Lipschitz continuous on $\Re^n$, then it is almost
everywhere differentiable, \ie the set $\Re^n\setminus C_G$ has measure zero, where $C_G$ is the subset of points in $\Re^n$ for which $G$
is differentiable. Hence, although the Jacobian of $\prox_{\gamma g}$ in the classical
sense might not exist everywhere, generalized differentiability notions, such as the $B$-subdifferential and the generalized Jacobian of Clarke,
can be employed to provide a local first-order approximation of $\prox_{\gamma g}$.
\begin{definition}\label{def:Jacs}
Let $G:\Re^n\to\Re^m$ be locally Lipschitz continuous at  $x\in\Re^n$. The B-subdifferential (or limiting Jacobian) of $G$ at $x$ is
\begin{equation*}
\partial_B G(x)\eqdef\left\{H\in\Re^{m\times n}\left|\right.\exists\ \{x^k\}\subset C_G\textrm{ with }x^k\to x, \nabla G(x^k)\to H\right\},
\end{equation*} 
whereas the (Clarke) generalized Jacobian of $G$ at $x$  is
$$\partial_C G(x)\eqdef\conv(\partial_BG(x)).$$
\end{definition}
If $G:\Re^n\to\Re^m$ is locally Lipschitz on $\Re^n$ then  $\partial_CG(x)$ is a nonempty, convex and compact subset of $m$ by $n$
matrices, and as a set-valued mapping it is outer-semicontinuous at every $x\in\Re^n$.
The next theorem shows that the elements of the generalized Jacobian of the proximal mapping are symmetric and positive semidefinite.
Furthermore, it provides a bound on the magnitude of their eigenvalues.
\begin{theorem}\label{th:JacProx}
Suppose that $g\in\mathcal{S}^0(\Re^n)$ and $x\in\Re^n$. Every $P\in\partial_C(\prox_{\gamma g})(x)$ is a symmetric positive semidefinite matrix that satisfies $\|P\|\leq 1$.
\end{theorem}
\begin{proof}
Since $g$ is convex, its Moreau envelope is a convex function as well, therefore every element of $\partial_C(\nabla g^\gamma)(x)$ is a
symmetric positive semidefinite matrix (see e.g. \cite[Sec.~8.3.3]{facchinei2003finite}). Due to~\eqref{eq:nabla_e}, we have that
$\prox_{\gamma g}(x)=x-\gamma\nabla g^\gamma(x)$,
therefore 
\begin{equation}\label{eq:JacProx}
\partial_C(\prox_{\gamma g})(x)= I-\gamma \partial_C(\nabla g^\gamma)(x).
\end{equation}
The last relation holds with equality (as opposed to inclusion in the general case) due to the fact that one of the summands is continuously differentiable.
Now from~\eqref{eq:JacProx} we easily infer that every element of $\partial_C(\prox_{\gamma g})(x)$ is a symmetric matrix.
Since $\nabla g^\gamma(x)$ is Lipschitz continuous with Lipschitz constant $\gamma^{-1}$, using \cite[Prop.~2.6.2(d)]{clarke1990optimization},
we infer that every $H\in \partial_C(\nabla g^\gamma)(x)$ satisfies $\|H\|\leq\gamma^{-1}$. Now, according to \eqref{eq:JacProx}, it holds
$$P\in\partial_C(\prox_{\gamma g})(x)\iff P=I-\gamma H,\quad H\in\partial_C(\nabla g^\gamma)(x).$$
Therefore,
$$ d'Pd=\|d\|^2-\gamma d'Hd\geq\|d\|^2-\gamma\gamma^{-1}\|d\|^2 = 0,\quad\forall P\in\partial_C(\prox_{\gamma g})(x).$$
On the other hand, since $\prox_{\gamma g}$ is Lipschitz continuous with Lipschitz constant 1, using \cite[Prop.~2.6.2(d)]{clarke1990optimization}
we obtain that $\|P\|\leq 1$, for all $P\in\partial_C(\prox_{\gamma g})(x)$.
\iftoggle{svver}{\qed}{}
\end{proof}

An interesting property of  $\partial_C\prox_{\gamma g}$, documented in the following proposition,
is useful whenever $g$ is (block) separable, \ie $g(x)=\sum_{i=1}^N g_i(x_i)$, $x_i\in\Re^{n_i}$, $\sum_{i=1}^N n_i=n$.
In such cases every $P\in\partial_C(\prox_{\gamma g})(x)$ is a (block)
diagonal matrix. This has favorable computational implications especially for large-scale problems.
For example, if $g$ is the $\ell_1$ norm or the indicator function of a box,
then the elements of $\partial_C\prox_{\gamma g}(x)$ (or $\partial_B\prox_{\gamma g}(x)$) are
diagonal matrices with diagonal elements in $[0,1]$ (or in $\{0,1\}$).
\begin{proposition}[separability]\label{prop:Separ}
If $g:\Re^n\to\overline{\Re}$ is (block) separable then every element of $\partial_B(\prox_{\gamma g})(x)$ and $\partial_C(\prox_{\gamma g})(x)$ is (block) diagonal.
\end{proposition}
\begin{proof}
Since $g$ is block separable, its proximal mapping has the form 
$$\prox_{\gamma g}(x)=(\prox_{\gamma g_1}(x_1),\ldots,\prox_{\gamma g_N}(x_N)).$$ 
The result follows directly by Definition~\ref{def:Jacs}.
\iftoggle{svver}{\qed}{}
\end{proof}

The following proposition provides a connection between the generalized Jacobian of the proximal mapping for a convex function and that of its conjugate, stemming from the celebrated Moreau's decomposition~\cite[Th. 14.3]{bauschke2011convex}.
\begin{proposition}[Moreau's decomposition]\label{prop:MorDec}
Suppose that $g\in\mathcal{S}^0(\Re^n)$. Then
\begin{align*}
\partial_B(\prox_{\gamma g^\star})(x)&=\{P=I-Q\left|\right.Q\in\partial_B(\prox_{g/\gamma})(x/\gamma)\},\\
\partial_C(\prox_{\gamma g^\star})(x)&=\{P=I-Q\left|\right.Q\in\partial_C(\prox_{g/\gamma})(x/\gamma)\}.
\end{align*}
\end{proposition}
\begin{proof}
Using Moreau's decomposition we have
$$\prox_{\gamma g^\star}(x)=x-\gamma\prox_{g/\gamma}(x/\gamma).$$
The first result follows directly by Definition~\ref{def:Jacs}, since $\prox_{\gamma g^\star}$ is expressed as the difference of two functions, one of which is continuously differentiable.
The second result follows from the fact that, with a little abuse of notation,
$$\conv\{I-Q\left|\right.Q\in\partial_B(\prox_{g/\gamma})(x/\gamma)\} = I-\conv(\partial_B(\prox_{g/\gamma})(x/\gamma)).$$
\iftoggle{svver}{\qed}{}
\end{proof}

\emph{Semismooth} mappings~\cite{qi1993nonsmooth} are precisely Lipschitz continuous mappings  for which the generalized Jacobian (and consequenlty the $B$-subdifferential) furnishes a first-order  approximation. 
\begin{definition}\label{def:Semismooth}
Let $G:\Re^n\to\Re^m$ be locally Lipschitz continuous 
at ${x}$. We say that $G$ is \emph{semismooth} at $\bar{x}$ if 
$$\|G(x)+H(\bar{x}-x)-G(\bar{x})\|=o(\|x-\bar{x}\|)\ \textrm{as}\ x\to\bar{x},\ \forall H\in\partial_CG(x)$$
whereas $G$ is said to be \emph{strongly semismooth} if $o(\|x-\bar{x}\|)$ can be replaced with  $O(\|x-\bar{x}\|^2)$.
\end{definition}
We remark that the original definition of semismoothness given by~\cite{mifflin1977semismooth} requires $G$ to be directionally differentiable at $x$.  The definition given here is the one employed by~\cite{gowda2004inverse}. Another worth spent remark is that $\partial_C G(x)$ can be replaced with the smaller set $\partial_B G(x)$ in Definition~\ref{def:Semismooth}.

Fortunately, the class of semismooth mappings is rich enough to include proximal mappings of most of the functions arising in interesting applications. For example \emph{piecewise smooth ($PC^1$) mappings}  are semismooth everywhere. Recall that a continuous mapping $G:\Re^n\to\Re^m$ is $PC^1$ if there exists a finite collection of smooth mappings $G_i:\Re^n\to\Re$, $i=1,\ldots,N$ such that
$$G(x)\in\{G_1(x),\ldots,G_N(x)\},\quad\forall x\in\Re^n.$$
The definition of $PC^1$ mappings given here is less general than the one of, e.g.,~\cite[Ch. 4]{scholtes2012introduction} but it suffices for our purposes. 
For every $x\in\Re^n$ we introduce the set of essentially active indices
$$I_G^e(x)=\{i\in[N]\ |\ x\in\cl(\Int\{x\ |\ G(x)=G_i(x)\})\}\footnote{$[N]\eqdef\{1,\ldots,N\}$ for any positive integer $N$.}.$$
In other words, $I_G^e(x)$ contains only indices of the pieces $G_i$ for which there exists a full-dimensional set on which $G$ agrees with $G_i$. In accordance to Definition~\ref{def:Jacs}, the generalized Jacobian of $G$ at $x$ is the convex hull of the Jacobians of the essentially active pieces, \ie \cite[Prop.~4.3.1]{scholtes2012introduction}
\begin{equation}\label{eq:PCJac}
\partial_C G(x)=\conv\{\nabla G_i(x)\ |\ i\in I_G^e(x)\}.
\end{equation}
As it will be clear in Section~\ref{sec:Examples}, in many interesting cases $\prox_{\gamma g}$ is $PC^1$ and thus semismooth. Furthermore, through~\eqref{eq:PCJac} an element of $\partial_C\prox_{\gamma g}(x)$ can be easily computed once $\prox_{\gamma g}(x)$ has been computed. 

A special but important class of convex functions  whose proximal mapping is $PC^1$ are piecewise quadratic (PWQ) functions. A convex function $g\in\mathcal{S}^0(\Re^n)$ is called PWQ if $\dom g$ can be represented as the union
of finitely many polyhedral sets, relative to each of which $g(x)$ is given by an expression of the form $(1/2)x'Qx+q'x+c$ ($Q\in\Re^{n\times n}$ must necessarily be symmetric positive semidefinite)~\cite[Def.~10.20]{rockafellar2011variational}. The class of PWQ functions is quite general since it includes e.g. polyhedral norms, indicators and support functions of polyhedral sets, and it is closed under addition, composition with affine mappings, conjugation, inf-convolution and inf-projection~\cite[Prop.~10.22, Proposition~11.32]{rockafellar2011variational}. It turns out that the proximal mapping of a PWQ function is \emph{piecewise affine} (PWA)~\cite[12.30]{rockafellar2011variational} ($\Re^n$ is partitioned in polyhedral sets relative to each of which $\prox_{\gamma g}$ is an affine mapping), hence  strongly semismooth~\cite[Prop.~7.4.7]{facchinei2003finite}. 
Another example of a proximal mapping that it is strongly semismooth is the projection operator over symmetric cones~ \cite{sun2002semismooth}.
We refer the reader to \cite{mifflin1999properties,meng2008lagrangian,meng2009moreau,meng2005semismoothness} for conditions that guarantee semismoothness of the proximal mapping for more general convex functions.

\subsection{Approximate generalized Hessian for $F_\gamma$}
Having established properties of generalized Jacobians for proximal mappings, we are now in position to construct a generalized Hessian for $F_\gamma$ that will allow the development of Newton-like methods with fast asymptotic convergence rates. The obvious route to follow is to assume that $\nabla F_\gamma$ is semismooth and employ $\partial_C(\nabla F_\gamma)$ as a generalized Hessian for $F_\gamma$. However, semismoothness would require extra assumptions on $f$.  Furthermore, the form of $\partial_C(\nabla F_\gamma)$ is quite complicated involving third-order partial derivatives of $f$. 
On the other hand, what is really needed to devise Newton-like algorithms with fast local convergence rates is a \emph{linear Newton approximation (LNA)}, cf. Definition~\ref{def:LNA},  at some stationary point of $F_\gamma$, which by Theorem \ref{Th:PropFg}\eqref{cor:Equivalence} is also a minimizer of $F$, provided that $\gamma\in (0,1/L_f)$. 
The approach we follow is largely based on \cite{sun1997computable}, \cite[Prop.~10.4.4]{facchinei2003finite}.
The following definition is taken from \cite[Def.~7.5.13]{facchinei2003finite}.

\begin{definition}\label{def:LNA}
Let $G:\Re^n\to\Re^m$ be continuous on $\Re^n$. We say that $G$ admits a linear Newton approximation at a vector $\bar{x}\in\Re^n$ if there exists a set-valued mapping $\mathscr{G}:\Re^n\rightrightarrows\Re^{n\times m}$ that has nonempty compact images, is upper semicontinuous at $\bar{x}$ and for any $H\in\mathscr{G}(x)$ 
$$\|G(x)+H(\bar{x}-x)-G(\bar{x})\|= o(\|x-\bar{x}\|)\ \textrm{ as } x\to\bar{x}.$$
If instead
$$\|G(x)+H(\bar{x}-x)-G(\bar{x})\|= O(\|x-\bar{x}\|^2)\ \textrm{ as } x\to\bar{x},$$
then we say that $G$ admits a strong linear Newton approximation at $\bar{x}$.
\end{definition}

Arguably the most notable example of a LNA for semismooth mappings is the generalized Jacobian, cf. Definition~\ref{def:Jacs}. However, semismooth mappings can admit LNAs different from the generalized Jacobian. More importantly, mappings that are not semismooth may also admit a LNA.
It turns out that we can define a LNA for $\nabla F_\gamma$ at any stationary
point, whose elements have a simpler form than those of $\partial_C(\nabla F_\gamma)$,
without assuming semismoothness of  $\nabla F_\gamma$. We call it \emph{approximate
generalized Hessian} and it is given by
\begin{equation*}\label{eq:LNAF}
\hat{\partial}^2 F_\gamma(x)\eqdef\{\gamma^{-1}(I-\gamma\nabla^2 f(x))(I-P(I-\gamma\nabla^2 f(x)))\ |\ P\in\partial_C(\prox_{\gamma g})(x-\gamma\nabla f(x))\}.
\end{equation*}
The key idea in the definition of $\hat{\partial}^2 F_\gamma$, reminiscent to the Gauss-Newton method for nonlinear least-squares problems, is to omit terms vanishing at $x_\star$ that contain third-order derivatives of $f$. The following proposition shows that $\hat{\partial}^2 F_\gamma$ is indeed a LNA of $\nabla F_\gamma$ at any $x_\star\in X_\star$.

\begin{proposition}\label{prop:LNAprops1}
Let $T(x)=x-\gamma\nabla f(x)$, $\gamma\in(0,1/L_f)$ and $x_\star\in X_\star$. Then
\begin{enumerate}[\rm (i)]
\item if $\prox_{\gamma g}$ is semismooth at $T({x}_\star)$, then $\hat{\partial}^2 F_\gamma$
is a LNA for $\nabla F_\gamma$ at ${x}_\star$,
\item if $\prox_{\gamma g}$ is strongly semismooth at $T({x}_\star)$, and $\nabla^2 f$
is locally Lipschitz around $x_\star$, then $\hat{\partial}^2 F_\gamma$
is a strong LNA for $\nabla F_\gamma$ at ${x}_\star$.
\end{enumerate}
\end{proposition}
\begin{proof}
See Appendix.
\end{proof}

The next proposition shows that every element of $\hat{\partial}^2 F_\gamma(x)$ is a symmetric positive semidefinite matrix, whose eigenvalues are lower and upper bounded uniformly over all $x\in\Re^n$.
\begin{proposition}\label{prop:PSDHess}
Any $H\in\hat{\partial}^2 F_\gamma(x)$ is symmetric positive semidefinite and satisfies 
\begin{equation}
\xi_1\|d\|^2\leq d'Hd\leq \xi_2\|d\|^2,\ \forall d\in\Re^n,
\end{equation}
where 
$\xi_1\eqdef\min\left\{(1-\gamma\mu_f)\mu_f,(1-\gamma L_f)L_f\right\}$, 
$\xi_2\eqdef\gamma^{-1}(1-\gamma\mu_f)$.
\end{proposition}
\begin{proof}
See Appendix.
\end{proof}
The next lemma shows uniqueness of the solution of~\eqref{eq:GenProb} under a nonsingularity assumption on the elements of $\hat{\partial}^2F_\gamma(x_\star)$. Its proof is similar to~\cite[Lem.~7.2.10]{facchinei2003finite}, however $\nabla F_\gamma$ is not required to be locally Lipschitz around $x_\star$.
\begin{lemma}\label{lem:sharpMin}
Let $x_\star\in X_\star$. Suppose that $\gamma\in(0,1/L_f)$,  $\prox_{\gamma g}$ is semismooth at $x_\star-\nabla f(x_\star)$ and every element of $\hat{\partial}^2F_\gamma(x_\star)$ is nonsingular. Then $x_\star$ is the unique solution of~\eqref{eq:GenProb}. In fact, there exist positive constants $\delta$ and $c$ such that 
$$\|x-x_\star\|\leq c\|G_\gamma(x)\|,\ \mathrm{for\ all\ } x\ \mathrm{with}\ \|x-x_\star\|\leq\delta.$$
\end{lemma}
\begin{proof}
See Appendix.
\end{proof}

\section{Forward-Backward Newton-CG Methods}\label{sec:FBNCG}
Having established the equivalence between   minimizing $F$ and $F_\gamma$, as well as a LNA for $\nabla F_\gamma$, it is now very easy to design globally convergent Newton-like algorithms with fast asymptotic convergence rates, for computing a $x_\star\in X_\star$. Algorithm~\ref{al:PNM} is a standard line-search method for minimizing $F_\gamma$, where a conjugate gradient method is employed to solve (approximately) the corresponding regularized Newton system. Therefore our algorithm does not require to form an element of the generalized Hessian of $F_\gamma$ explicitly. It only requires the computation of the corresponding matrix-vector product and is thus suitable for large-scale problems. Similarly, there is no need to form explicitly the Hessian of $f$, in order to compute the  directional derivative $\nabla F_\gamma(x^k)'d^k$ needed in the backtracking procedure for computing the stepsize~\eqref{eq:Armijo}; only matrix-vector products with $\nabla^2 f(x)$ are required.
Under nonsingularity of the elements of $\hat{\partial}^2 F_\gamma(x_\star)$,  eventually the stepsize becomes equal to 1 and Algorithm~\ref{al:PNM} reduces to a regularized version of the (undamped) linear Newton method \cite[Alg. 7.5.14]{facchinei2003finite} for solving $\nabla F_\gamma(x)=0$.
\begin{algorithm}
\LinesNumbered
\DontPrintSemicolon
\caption{Forward-Backward Newton-CG Method (FBN-CG I)} \label{al:PNM}
\KwIn{$\gamma\in (0,1/L_{f})$, $\sigma\in \left(0,1/2\right)$, $\bar{\eta}\in (0,1)$, $\zeta\in (0,1)$, $\rho\in(0,1]$, $x^0\in\Re^n$, $k=0$}

 Select a $H^ k\in \hat{\partial}^2 F_{\gamma}(x^ k)$. Apply CG to
 \begin{equation}\label{eq:RegNewtSys}
(H^ k+\delta_ k I)d^ k=-\nabla F_\gamma(x^ k)
\end{equation}
to compute a $d^ k\in\Re^n$ that satisfies
\begin{equation}\label{eq:InRegNewtSys}
\|(H^ k+\delta_ k I)d^ k+\nabla F_\gamma(x^ k)\|\leq \eta_ k\|\nabla F_\gamma(x^ k)\|,
\end{equation}
where
\begin{subequations}
\begin{align}
\delta_ k&=\zeta\|\nabla F_\gamma(x^ k)\|,\label{eq:deltas}\\
\eta_ k&=\min\{\bar{\eta},\|\nabla F_\gamma(x^ k)\|^\rho\}.\label{eq:etas}
\end{align}
\end{subequations}\;
Compute 
$\tau_ k=\max\{2^{-i}\ |\ i=0,1,2,\ldots\}$ such that
\begin{equation}\label{eq:Armijo}
F_{\gamma}({x}^{ k}+\tau_ k d^ k)\leq F_{\gamma}({x}^{ k})+\sigma \tau_ k\nabla F_{\gamma}({x}^{ k})'d^{ k}.
\end{equation}\;
${x}^{ k+1}\gets {x}^{ k}+\tau_ k d^ k$\;
$ k\gets k+1$ and go to Step 1.\;
\end{algorithm}

The next theorem delineates the basic convergence properties of Algorithm~\ref{al:PNM}.
\begin{theorem}\label{th:ConvAlg1}
Every accumulation point of the sequence $\{x^ k\}$ generated by Algorithm~\ref{al:PNM} belongs to $X_\star$.
\end{theorem}
\begin{proof}
We will first show that the sequence $\{d^ k\}$ is \emph{gradient related to $\{x^ k\}$} \cite[Sec. 1.2]{bertsekas1999nonlinear}. That is, for any subsequence 
$\{x^ k\}_{ k\in \mathcal N}$ that converges to a nonstationary point of $F_\gamma$, \ie
\begin{equation}\label{eq:gradRel0}
\lim_{ k\to\infty, k\in\NN} \|\nabla F_\gamma(x^ k)\|=\kappa\neq 0,
\end{equation}
the corresponding subsequence $\{d^ k\}_{ k\in\NN}$ is bounded and satisfies
\begin{equation}\label{eq:gradRel}
\limsup_{ k\to\infty, k\in\NN} \nabla F_\gamma(x^ k)'d^{ k}<0.
\end{equation}
Without loss of generality we can restrict to subsequences for which $\nabla F_\gamma(x^ k)\neq 0$, for all $ k\in \mathcal N$.
Suppose that $\{x^ k\}_{ k\in \mathcal N}$ is one such subsequence.
Due to~\eqref{eq:deltas}, we have $\delta_ k>0$ for all $ k\in \mathcal N$. 
Matrix $H^k$ is positive semidefinite due to Proposition~\ref{prop:PSDHess},
therefore $H^ k+\delta_ k I$ is nonsingular for all $k\in \mathcal N$ and
$$\|(H^ k+\delta_ k I)^{-1}\|\leq\delta_ k^{-1}=\frac{1}{\zeta\|\nabla F_\gamma(x^ k)\|}.$$
Now, direction $d^ k$ satisfies
\begin{equation*}
d^ k=(H^ k+\delta_ k I)^{-1}(r^ k-\nabla F_\gamma(x^ k)),
\end{equation*}
where $r^ k=(H^ k+\delta_ k I)d^ k+\nabla F_\gamma(x^ k)$.
Therefore
\begin{align}
\|d^ k\|&\leq\|(H^ k+\delta_ k I)^{-1}\|(\|r^ k\|+\|\nabla F_\gamma(x^ k)\|)\label{eq:boundd}\\
&\leq\frac{1}{\zeta\|\nabla F_\gamma(x^ k)\|}(\eta_ k\|\nabla F_\gamma(x^ k)\|+\|\nabla F_\gamma(x^ k)\|)\leq(1+\bar{\eta})/\zeta,\nonumber
\end{align}
proving that $\{d^ k\}_{ k\in\mathcal{N}}$ is bounded. 
According to \cite[Lemma A.2]{dembo1983truncated}, when CG is applied to~\eqref{eq:RegNewtSys} we have that 
\begin{equation}\label{eq:CGprop}
\nabla F_\gamma(x^ k)'d^ k\leq-\frac{1}{\|H^ k+\delta_ k I\|}\|\nabla F_\gamma(x^ k)\|^2.
\end{equation}
Using~\eqref{eq:deltas}
and  Proposition~\ref{prop:PSDHess}, we have that 
$$\|H^ k+\delta_ k I\|\leq \gamma^{-1}+\zeta\| \nabla F_\gamma(x^ k)\|,$$
therefore 
\begin{equation}\label{eq:SuffDec}
\nabla F_\gamma(x^ k)'d^ k\leq-\frac{\|\nabla F_\gamma(x^ k)\|^2}{ \gamma^{-1}+\zeta\|\nabla F_\gamma(x^ k)\|},\quad\forall  k\in\NN,
\end{equation}
As $ k(\in\NN)\to\infty$, the right hand side of~\eqref{eq:SuffDec} converges to $-\kappa^2/(\gamma^{-1}+\zeta\kappa)$, which is either a finite negative number (if $\kappa$ is finite) or $-\infty$. In any case, this together with~\eqref{eq:SuffDec} confirm that~\eqref{eq:gradRel} is valid as well, proving that $\{d^ k\}$ is gradient related to $\{x^ k\}$. All the assumptions of \cite[Prop.~1.2.1]{bertsekas1999nonlinear} hold, therefore every accumulation point of $\{x^ k\}$ converges to a stationary point of $F_\gamma$, which by Theorem~\ref{Th:PropFg}(\ref{cor:Equivalence}) is also a minimizer of $F$.
\iftoggle{svver}{\qed}{}
\end{proof}

The next theorem shows that under a nonsingularity assumption on $\hat{\partial}^2 F_\gamma(x_\star)$, the asymptotic rate of convergence of the sequence generated by Algorithm~\ref{al:PNM} is at least superlinear.

\begin{theorem}\label{eq:PNMconvRate}
Suppose that $x_\star$ is an accumulation point of the sequence $\{x^ k\}$ generated by Algorithm~\ref{al:PNM}. If $\prox_{\gamma g}$  is semismooth at $x_\star-\gamma\nabla f(x_\star)$  and every element of $\hat{\partial}^2 F_\gamma(x_\star)$ is nonsingular, then the entire sequence converges to $x_\star$ and the convergence rate is Q-superlinear. Furthermore, if $\prox_{\gamma g}$  is strongly semismooth at $x_\star-\gamma\nabla f(x_\star)$ and $\nabla^2 f$ is locally Lipschitz continuous around $x_\star$  then $\{x^k\}$ converges to $x_\star$ with Q-order at least $\rho$.
\end{theorem}
\begin{proof}
Theorem~\ref{th:ConvAlg1} asserts that $x_\star$ must be a stationary point for  $F_\gamma$. Due to Proposition~\ref{prop:LNAprops1}, $\hat{\partial}^2 F_\gamma$  is a LNA of $\nabla F_\gamma$ at $x_\star$.
Due to Lemma~\ref{lem:sharpMin}, $x_\star$ is the globally unique minimizer of $F$. Therefore, by Theorem~\ref{th:ConvAlg1} every subsequence must converge to this unique accumulation point, implying that the entire sequence converges to $x_\star$.
Furthermore, for any $ k$
\begin{align}
\|\nabla F_\gamma(x^ k)\|&\leq \|I-\gamma\nabla^2 f(x^ k)\|\|G_\gamma(x^ k)\|\nonumber\\
&\leq \|G_\gamma(x^ k)-G_\gamma(x_\star)\|\leq 2\gamma^{-1}\|x^ k-x_\star\|,\label{eq:Calmness}
\end{align}
where the second inequality follows from $G_\gamma(x_\star)=0$ and Lemma~\ref{le:zNonExp} (in the Appendix).

 We know that $d^ k$ satisfies $(H^ k+\delta_ k I) d^ k+\nabla  F_\gamma(x^ k)=r^ k$. Therefore, for sufficiently large $ k$, we have
\small
\begin{align}
\|x^ k+d^ k-x_\star\|&=\|x^ k+(H^ k+\delta_ k I)^{-1}(r^ k-\nabla F_\gamma(x^ k))-x_\star\|\nonumber\\
& =\|(H^ k+\delta_ k I)^{-1}(H^ k(x^ k-x_\star)-\nabla F_\gamma(x^ k)+ \delta_ k(x^ k-x_\star)+r^ k)\|\nonumber\\
& \leq \|(H^ k+\delta_ k I)^{-1}\|\left(\|H^ k(x^ k-x_\star)+\nabla F_\gamma(x_\star)-\nabla F_\gamma(x^ k)\|\right.\nonumber\\
& \phantom{\leq \|(H^ k+\delta_ k I)^{-1}\|\left(\right.} \left.+~\delta_ k\|x^ k-x_\star\|+\|r^ k\|\right)\nonumber\\
& \leq\kappa\left(\|H^ k(x^ k-x_\star)+\nabla F_\gamma(x_\star)-\nabla F_\gamma(x^ k)\|\right.\nonumber\\
& \phantom{\leq\kappa\left(\right.}\left.+~2\zeta\gamma^{-1}\|x^ k-x_\star\|^2+\eta\gamma^{-1}\|x-x_\star\|^{1+\rho}\right)\label{eq:convRateEq}
\end{align}
\normalsize
where the last inequality follows by~\eqref{eq:deltas},~\eqref{eq:etas},~\eqref{eq:Calmness}. 
Therefore, since $\hat{\partial}^2 F_\gamma$ is a LNA of $\nabla F_\gamma$ at $x_\star$, we have
\begin{equation}\label{eq:Qsup}
\|x^ k+d^ k-x_\star\|=o(\|x^ k-x_\star\|),
\end{equation}
 while if it is a strong LNA we have 
\begin{equation}\label{eq:Qquad}
\|x^ k+d^ k-x_\star\|=O(\|x^ k-x_\star\|^{1+\rho}).
\end{equation}
In other words, $\{d^ k\}$ is \emph{superlinearly convergent with respect to} $\{x^ k\}$~\cite[Sec. 7.5]{facchinei2003finite}. 
Eventually, we have
\begin{align}
\nabla F_\gamma(x^k)'d^k+{d^k}'(H^k+\delta_kI)d^k&\leq\eta_k\|\nabla F_\gamma(x^k)\|\|d^k\|\leq\|\nabla F_\gamma(x^k)\|^{\rho+1}\|d^k\|\nonumber\\
&\leq2\gamma^{-(\rho+1)}\|x^k-x_\star\|^{\rho+1}\|d^k\|\nonumber\\
&=O(\|d^k\|^{\rho+2}),\label{eq:unitStepBasic}
\end{align}
where the first inequality follows by~\eqref{eq:InRegNewtSys}, the second by~\eqref{eq:etas}, the third inequality follows by~\eqref{eq:Calmness} and the equality follows from the fact that $\{d^ k\}$ is superlinearly convergent with respect to $\{x^ k\}$, which implies $\|x^k-x_\star\|=O(\|d^k\|)$~\cite[Lem.~7.5.7]{facchinei2003finite}.

Since $\hat{\partial}^2 F_\gamma$ is a LNA of $\nabla F_\gamma$ at $x_\star$, it has nonempty compact images and is upper semicontinuous at $x_\star$. This, together with the fact that $\{x^ k\}$ converges to $x_\star$ and the nonsingularity assumption on the elements of $\hat{\partial}^2 F_\gamma(x_\star)$ imply through \cite[Lem.~7.5.2]{facchinei2003finite} that for sufficiently large $ k$, $H^ k$ is nonsingular and there exists a $\kappa>0$ such that
$$\max\{\|H^ k\|,\|H^ k\|^{-1}\}\leq\kappa.$$
Therefore, eventually we have $\lambda_{\min}(H^ k+\delta_ k I)\geq\lambda_{\min}(H^ k)\geq\kappa$.
The last inequality together with~\eqref{eq:unitStepBasic} imply that there exists a $\theta>0$ such that eventually
\begin{equation}\label{eq:unitStep}
\nabla F_\gamma(x^k)'d^k\leq -\theta \|d^k\|^2.
\end{equation}
Following the same line of proof as in~\cite[Prop.~7.4.10]{facchinei2003finite},  it can be shown that
\begin{equation}\label{le:2ndOrd}
\lim_{\stackrel{\|d\|\to 0}{H\in\hat{\partial}^2 F_\gamma(x_\star+d)}}\frac{F_{\gamma}(x_\star+d)-F_{\gamma}(x_\star)-\nabla F_\gamma(x_\star)'d-\tfrac{1}{2}d'Hd}{\|d\|^2}=0.
\end{equation}
We remark here that~\cite[Prop.~7.4.10]{facchinei2003finite} assumes semismoothness of $\nabla F_\gamma$ at $x_\star$ and proves~\eqref{le:2ndOrd} with $\partial_C(\nabla F_\gamma)$ in place of $\hat{\partial}^2 F_\gamma$, but exactly the same arguments apply for any LNA of $\nabla F_\gamma$ at $x_\star$ even without the semismoothness assumption.

Using~\eqref{eq:unitStep},~\eqref{le:2ndOrd} and exactly the same arguments as in the proof of~\cite[Prop.~8.3.18(d)]{facchinei2003finite} or~\cite[Th. 3.2]{facchinei1995minimization} we have that eventually
\begin{equation}
F_\gamma(x^ k+d^ k)\leq F_\gamma(x^ k)+\sigma\nabla F_\gamma(x^ k)'d^ k,
\end{equation}
which means that there exists a positive integer $\bar{k}$ such that $\tau_k=1$, for all $k\geq \bar{k}$. Therefore, for all  $k\geq \bar{k}$
$$x^{k+1}=x^k+d^k.$$
This together with~\eqref{eq:Qsup},~\eqref{eq:Qquad} proves the corresponding convergence rates for $\{x^k\}$.
\iftoggle{svver}{\qed}{}
\end{proof}

When $f$ is strongly convex quadratic, Theorem~\ref{th:ProxPropQuad} guarantees that $F_\gamma$ is strongly convex and we can give a complexity estimate for Algorithm~\ref{al:PNM}. In particular, the global convergence rate for the function values and the iterates is linear. 
\begin{theorem}\label{th:ComplPNM}
Suppose that $f$ is quadratic and $\mu_f>0$. If $\zeta=0$ then 
\begin{subequations}
\begin{align}
F(P_\gamma(x^k))-F_\star\leq r_{F_\gamma}(F_\gamma(x^0)-F_\star)),\label{eq:QuadRateF}\\
\|x^{k}-x_\star\|^2\leq \frac{L_{F_\gamma}}{\mu_{F_\gamma}}r_{F_\gamma}^k\|x^0-x_\star\|^2\label{eq:QuadRatex}
\end{align}
\end{subequations}
where $r_{F_\gamma}= 1-2\left(\frac{\mu_{F_\gamma}}{L_{F_\gamma}}\right)^3\frac{\sigma(1-\sigma)}{1+\eta}$.
\end{theorem}
\begin{proof}
See Appendix.
\end{proof}

Algorithm~\ref{al:PNM} exhibits fast asymptotic convergence rates provided that the elements of $\hat{\partial}^2 F_\gamma(x_\star)$ are nonsingular,
but not much can be said about its global convergence rate, unless $f$ is convex quadratic. Even in this favorable case the corresponding complexity
estimates are very loose due to the variable metric used by the algorithm, cf. Theorem~\ref{th:ComplPNM}.

Another reason for the failure to derive meaningful complexity estimates is the fact that Algorithm~\ref{al:PNM} ``forgets'' about the convex
structure of $F$, since it tries to minimize directly $F_\gamma$ which can be nonconvex and its gradient may not be globally Lipschitz continuous.
Specifically, Algorithm~~\ref{al:PNM} may fail to be a descent method for $F$ (although it satisfies that property for $F_\gamma$). Furthermore the
iterates $x^ k$ produced by Algorithm~\ref{al:PNM} may lie outside $\dom g$ (but $P_\gamma(x^ k)\in\dom g$, see Theorem~\ref{Th:PropFg}(\ref{prop:LowBnd})).
In this section, we show how Algorithm~\ref{al:PNM} can be modified so as to be able to derive global complexity estimates, similar to the ones for the
proximal gradient method, and at the same time retain fast asymptotic convergence rates. The key idea is to inject a forward-backward step after the
Newton step (cf. Alg.~\ref{al:PGNM}) and analyze the consequences of this choice on $F$, directly. This guarantees that the sequence of function values
for both $F$ and $F_\gamma$ are monotone nonincreasing.

\begin{algorithm}
\LinesNumbered
\DontPrintSemicolon
\caption{Forward-Backward Newton-CG Method II (FBN-CG II)} \label{al:PGNM}
\KwIn{$\gamma\in (0,1/L_{f})$, $\sigma\in \left(0,{1}/{2}\right)$, $\mathcal{K}\subseteq\Nn$, $ k=0$,  $s_{0}=0$, $x^{0}\in\dom g$}
\uIf{$ k\in\mathcal{K}$ or $s_{ k}=1$}{
Execute steps 1 and 2 of Algorithm~\ref{al:PNM} to compute direction $d^k$ and step $\tau_k$\;
$\hat{x}^ k\gets {x}^{ k}+\tau_ k d^ k$\;
\lIf{$\tau_ k=1$}{$s_{ k+1}\gets 1$} \lElse{$s_{ k+1}\gets 0$}
}
\Else{
$\hat{x}^{ k}\gets x^{ k}$, $s_{ k+1}\gets 0$
}
$x^{ k+1}\gets \prox_{\gamma g}(\hat{x}^ k-\gamma\nabla f(\hat{x}^ k))$\;
$ k\leftarrow k+1$ and go to Step 1.\;
\end{algorithm}

We show below that the sequence of iterates $\{x^ k\}_{k\in\Nn}$ produced by Algorithm~\ref{al:PGNM} enjoys the same favorable properties in terms of convergence and local convergence rates, as the one of Algorithm~\ref{al:PNM}.

\begin{theorem}
Every accumulation point of the sequence $\{x^ k\}$ generated by Algorithm~\ref{al:PGNM} belongs to $X_\star$.
\end{theorem}
\begin{proof}
If $\mathcal{K}=\emptyset$ then Algorithm~\ref{al:PGNM} is equivalent to FBS and the result has been already proved in \cite[Th. 1.2]{beck2010gradient}. Let us then assume
$\mathcal{K}\neq\emptyset$ and distinguish between two cases. First, we deal with the case where $ k\notin\mathcal{K}$ and $s_{ k}=0$. 
Putting $x=\bar{x}=x^{ k}$ in~\eqref{eq:ProxBasic} we obtain
\begin{equation}\label{eq:DesPGNM}
F(x^{ k+1})-F(x^{ k})\leq -\tfrac{\gamma}{2}\|G_{\gamma}(x^{ k})\|^2.
\end{equation}
For the case where $ k\in\mathcal{K}$ or $s_{ k}=1$, unless $\nabla F_{\gamma}(x^ k)=0$ (which means that  $x^ k$ is a minimizer of $F$), we have $F_{\gamma}(\hat{x}^ k)< F_{\gamma}(x^ k)$ due to~\eqref{eq:Armijo}. Using parts \eqref{prop:UppBnd} and~\eqref{prop:LowBnd} of Theorem~\ref{Th:PropFg} we obtain
\begin{align*}
F(x^{ k+1})&=F(P_\gamma(\hat{x}^ k))\leq F_\gamma(\hat{x}^ k)\\
&\leq F_\gamma(x^ k)\leq F(x^ k)-\tfrac{\gamma}{2}\|G_{\gamma}(x^ k)\|^2
\end{align*}
and again we arrive at \eqref{eq:DesPGNM}.
 
Summing up, Eq.~\eqref{eq:DesPGNM} is satisfied for every $k\in\Nn$.
Since $\{F(x^k)\}$ is monotonically nonincreasing, it converges to a finite value
(since we have assumed that $F$ is proper),
therefore $\{F(x^k)-F(x^{k+1})\}$
converges to zero. This implies through~\eqref{eq:DesPGNM} that
$\{\|G_{\gamma}(x^{ k})\|^2\}$ converges to zero. Since
$\|G_{\gamma}({\cdot})\|^2$ is a continuous nonnegative function which becomes
zero if and only if $x\in X_\star$, it follows that
every accumulation point of $\{x^{ k}\}$ belongs to $X_\star$.
\iftoggle{svver}{\qed}{}
\end{proof}

\begin{theorem}\label{eq:PGNMconvRate}
Suppose $\mathcal{K}$ is infinite. Under the assumptions of Theorem~\ref{eq:PNMconvRate} the same results
apply also to the sequence of iterates produced by Algorithm~\ref{al:PGNM}.
\end{theorem}
\begin{proof}
Following exactly the same steps as in the proof of Theorem~\ref{eq:PNMconvRate} we can show that $\{d^k\}$ is superlinearly convergent with respect to $\{x^k\}$. Indeed, the derivation is independent of the algorithmic scheme and it is only related to how the direction $d^k$ is generated. This means that unit stepsize is eventually accepted, \ie, there exists a positive integer $\bar{k}$ such that $s^k=1$ for all $k\geq\bar{k}$. Therefore, eventually the iterates are given by
$$x^{ k+1}=P_\gamma(x^ k+d^ k),\qquad k\geq\bar{k}.$$
Due to nonexpansiveness of $P_\gamma$ we have
$$\|x^{ k+1}-x_\star\|=\|P_\gamma(x^ k+d^ k)-P_\gamma(x_\star)\|\leq\|x^k+d^k-x_\star\|.$$
The proof finishes by invoking~\eqref{eq:convRateEq}.
\iftoggle{svver}{\qed}{}
\end{proof}
As the next theorem shows, Algorithm~\ref{al:PGNM} not only enjoys fast asymptotic convergence rate properties but also comes with the following global complexity estimate.
\begin{theorem}\label{th:PGNMbnds1}
Let $\{x^{ k}\}$ be a sequence generated by Algorithm \ref{al:PGNM}. Assume that the level sets of $F$ are bounded, \ie $\|x-x_\star\|\leq R$ for some $x_\star\in X_\star$ and all $x\in\Re^n$ with $F(x)\leq F(x^0)$. If $F(x^0)-F_\star\geq R^2/\gamma$ then 
\begin{equation}\label{eq:FirstStep}
F(x^1)-F_\star\leq\frac{R^2}{2\gamma}.
\end{equation} 
Otherwise, for any $ k\in\Nn$ we have
\begin{equation}\label{eq:kStep}
F(x^{ k})-F_\star\leq\frac{2R^2}{\gamma(k+2)}.
\end{equation}
\end{theorem}

\begin{proof}
See Appendix.
\end{proof}

When $f$ is strongly convex  the global rate of convergence is linear. The next theorem gives the corresponding complexity estimates. 
\begin{theorem}\label{th:PGNMbnds2}
If $f\in\mathcal{S}_{\mu_f,L_f}^{1,1}(\Re^n)$, $\mu_f>0$, then 
\begin{subequations}
\begin{align}
F\left(x^ k\right)-F_\star&\leq(1+\gamma\mu_f)^{- k}(F(x^0)-F_\star),\label{eq:strC1}\\
\|x^{ k+1}-x_\star\|^2&{\leq}\frac{1-\gamma\mu_f}{\gamma\mu_f(1+\gamma\mu_f)^{k}}\|x^{0}-x_\star\|^2.\label{eq:strC2}
\end{align}
\end{subequations}
\end{theorem}
\begin{proof}
See Appendix.
\end{proof}
\begin{remark}\label{re:LSgamma}
We should remark that Theorems~\ref{th:PGNMbnds1} and~\ref{th:PGNMbnds2} remain valid even if
$L_f$ (and thus $\gamma$) is unknown and instead a backtracking line search procedure similar to those described in~\cite{beck2009fast,nesterov2007gradient}, is performed to determine a suitable value for $\gamma$.
\end{remark}

\section{Examples}\label{sec:Examples}
In this section we discuss the generalized Jacobian of the proximal mapping of many relevant nonsmooth functions. Some of the considered examples will be particularly useful in
Section~\ref{sec:Simulations} to test the effectiveness of Algorithms~\ref{al:PNM} and \ref{al:PGNM} on specific problems.

\subsection{Indicator functions}\label{ex:IndFun} Constrained convex problems can be cast in the composite form \eqref{eq:GenProb} by encoding the feasible set $D$ with the appropriate indicator
function $\delta_D$. Whenever $\Pi_D$, the projection onto $D$, is efficiently computable, then algorithms like the forward-backward splitting \eqref{eq:FBS} can be conveniently considered. In the following we
analyze the generalized Jacobian of some of such projections.

\subsubsection{Affine sets}\label{ex:projAff}
If $D=\{x\ |\ Ax= b\}$, $A\in\Re^{m\times n}$, then $\Pi_D(x)=x-A^{\dagger} (Ax-b)$, where $A^{\dagger}$ is the Moore-Penrose pseudoinverse of $A$. For example  if $m<n$ and $A$ has full row rank, then $A^{\dagger}=A'(AA')^{-1}$. Obviously $\Pi_D$ is an affine mapping, thus  everywhere differentiable with
\begin{equation}
\partial_C(\Pi_D)(x)=\partial_B(\Pi_D)(x)=\{\nabla\Pi_D(x)\}=\{I-A^{\dagger} A\}.
\end{equation}
\subsubsection{Polyhedral sets}\label{ex:projPoly}
In this case $D=\{x\ |\ Ax=b,\ Cx\leq d\}$, with $A\in\Re^{m_1\times n}$ and $C\in\Re^{m_2\times n}$. It is well known 
that $\Pi_D$ is piecewise affine. In particular let 
$$\mathscr{I}_D=\left\{I\subseteq[m_2]\ \left|\begin{array}{l} \textrm{there exists a vector }x\in\Re^n\textrm{ with }Ax=b,\\ \ C_{i\cdot}x=d_i,\ i\in I,\ C_{j\cdot}x<d_j,\ j\in[m_2]\setminus I\end{array}\right.\right\}$$
For each $I\in \mathscr{I}_D$ let 
\begin{align*}
F_I&=\{x\in D\ |\ C_{i\cdot}x=d_i,\ i\in I\},\\
S_I&=\aff F_I=\{x\in \Re^n\ |\ Ax=b,\ C_{i\cdot}x=d_i,\ i\in I\},\\
N_I&=\cone\left\{\begin{bmatrix}A'&C_{I\cdot}'\end{bmatrix}\right\},\\
C_I&=F_I+N_I.
\end{align*}
We then have $\Pi_D(x)\in\{\Pi_{S_I}(x)\ |\ I\in\mathscr{I}_D\}$, \ie $\Pi_D$ is a piecewise affine function. The affine  pieces of $\Pi_D$ are the projections on the corresponding affine subspaces $S_I$, see Section~\ref{ex:projAff}.  In fact for each $x\in C_I$ we have $\Pi_D(x)=\Pi_{S_I}(x)$, each $C_I$ is full dimensional and $\Re^n=\bigcup_{I\in\mathscr{I}_D}C_I$. 
For each $I\in \mathscr{I}_D$ let $P_I=\nabla \Pi_{S_I}$ and for each $x\in\Re^n$ let $J(x)=\{I\in \mathscr{I}_D\ |\ x\in C_I\}$. Then
$$\partial_C(\Pi_D)(x)=\conv\partial_B(\Pi_D)(x)=\conv\{P_I\ | I\in J(x)\}.$$
Therefore, in order to determine an element $P$ of $\partial_B(\Pi_D)(x)$ it suffices to compute $\bar{x}=\Pi_D(x)$ and take 
$P=I-B^{\dagger} B$, where
$$B=\begin{bmatrix}A\\ C_{I(x)\cdot}\end{bmatrix},$$
and $I(x)=\{i\in[n]\ |\ A_{i\cdot}\bar{x}=b_i\}$.
\subsubsection{Halfspaces}
We denote $(x)_{+} = \max\{0, x\}$. If $D=\{x\ |\ a'x\leq b\}$ then
$$\Pi_D(x)=x-\left(\frac{(a'x-b)_+}{\|a\|_2^2}\right)a$$
and 
$$\partial_C(\Pi_D)(x)=\begin{cases}\{I-(1/\|a\|^2)aa'\},&\textrm{ if }a'x>b,\\
\{I\},&\textrm{ if }a'x<b,\\
\conv\{I,I-(1/\|a\|^2)aa'\},&\textrm{ if }a'x=b.
\end{cases}$$
\subsubsection{Boxes}\label{ex:projBox}
Consider the box $D=\{x\ |\ \ell\leq x\leq u\}$, with $\ell_i\leq u_i$. We have
$$\Pi_D(x)=\min\{\max\{x,\ell\},u\}.$$
The corresponding indicator function $\delta_D$ is clearly separable, therefore  (Prop.~\ref{prop:Separ}) every element $P\in \partial_B(\Pi_D)(x)$ is diagonal with 
$$P_{ii}=\begin{cases}
1,&\textrm{ if } \ell< x< u,\\
0,&\textrm{ if } x<\ell\textrm{ or }x> u,\\
\{0,1\}, &\textrm{ if } x=\ell\textrm{ or }x= u.
\end{cases}$$
\subsubsection{Unit simplex}\label{ex:projSimplex}
When $D=\left\{x\ |\ x\geq 0,\ \sum_{i=1}^nx_i=1\right\}$,
one can easily see, by writing down the optimality conditions for the corresponding projection problem, that
$$\Pi_D(x)=(x-\lambda\mathbf{1})_+,$$
where $\lambda$ solves $\mathbf{1}'(x-\lambda\mathbf{1})_+=1$. Since the unit simplex is a polyhedral set, we are dealing with a special case of Section~\ref{ex:projPoly}, where $A=\mathbf{1}_n'$, $b=1$, $C=-I_n$ and $d=0$. Therefore, to calculate an element of the generalized Jacobian of the projection, we first compute $\Pi_D(x)$ and then determine the set of active indices \mbox{$J=\{i\in[n]\ |\ (\Pi_D(x))_i=0\}$}. Let $n_J=|J|$ and $J_c=[n]\setminus J$. An element $P$ of $\partial_B(\Pi_D)(x)$ is given by
$$P_{ij}=\begin{cases}0,&\textrm{ if } i,j\in J\\
-1/(n-n_J),&\textrm{ if } i\neq j, i,j\in J_c,\\
1-1/(n-n_J),&\textrm{ if } i= j, i,j\in J_c.\end{cases}$$
Notice that $P$ is block-diagonal after a permutation of rows and columns. 
The nonzero part $P_{J_cJ_c}$ is Toeplitz, so we can compute matrix vector products in $O(n_{J_c}\log n_{J_c})$ instead of $O(n_{J_c}^2)$ operations. Computing an element of the generalized Jacobian of the projection on \mbox{$D=\{x\ |\ a'x=b,\ \ell\leq x\leq u\}$} can be treated in a similar fashion.
\subsubsection{Euclidean unit ball}
If $g=\delta_{B_2}$, where $B_2$ is the Euclidean unit ball then 
$$\Pi_{B_2}(x)=\begin{cases}
x/\|x\|_2,&\textrm{ if } \|x\|_2>1,\\
x, & \textrm{ otherwise }
\end{cases}$$
and
$$\partial_C(\Pi_{B_2})(x)=\begin{cases}
\{(1/\|x\|_2)(I-ww')\},&\textrm{ if } \|x\|_2>1,\\
\{I\},&\textrm{ if } \|x\|_2<1,\\
\conv\{(1/\|x\|_2)(I-ww'),I\}, & \textrm{ otherwise,}
\end{cases}$$
where $w=x/\|x\|_2^2$.
Equality follows from the fact that $\Pi_{B_2}:\Re^n\to\Re^n$ is a piecewise smooth function.
\subsubsection{Second-order cone}
Given a point $x=(x_0,\bar{x})\in\Re\times\Re^n$, each element of $V\in\partial_B(\Pi_K)(z)$ has the following
representation~\cite[Lem.~2.6]{kanzow2009local}:
$$ V=0\textrm{ or } V=I_{n+1}\textrm{ or } V=\begin{bmatrix}1 & \bar{w}'\\ \bar{w}& H\end{bmatrix}, $$
for some vector $\bar{w}\in\Re^n$ with $\|\bar{w}\|_2=1$ and some matrix $H\in\Re^{n\times n}$ of the form
\begin{equation}
H=(1+\alpha)I_n-\alpha\bar{w}\bar{w}',\quad |\alpha|\leq 1. \label{eq:HSOC}
\end{equation}
More precisely:
\begin{enumerate}[(i)]
\item if $x_0\neq\pm\|\bar{x}\|_2$, then $\bar{w} = \bar{x}/\|\bar{x}\|,\ \alpha=x_0/\|\bar{x}\|,$
\item if $\bar{x}\neq 0$ and ${x}_0=+\|\bar{x}\|_2$, then $\bar{w} = \bar{x}/\|\bar{x}\|,\ \alpha=+1,$
\item if $\bar{x}\neq 0$ and ${x}_0=-\|\bar{x}\|_2$, then $\bar{w} = \bar{x}/\|\bar{x}\|,\ \alpha=-1,$
\item if $\bar{x}=0$ and $x_0=0$, then either $V=0$ or $V=I_{n+1}$ or it has $H$ as in~\eqref{eq:HSOC}
for any $\bar{w}$ with $\|\bar{w}\|=1$ and $\alpha$ with $|\alpha|\leq 1$.
\end{enumerate}
\subsection{Vector norms} 
\subsubsection{Euclidean norm}
If $g(x)=\|x\|_2$ then the proximal mapping is given by
$$\prox_{\gamma g}(x)=\begin{cases}
(1-\gamma/\|x\|_2)x,&\textrm{ if } \|x\|_2\geq\gamma,\\
0,&\textrm{ otherwise}.\end{cases}$$
Since $\prox_{\gamma g}$ is a $P{C}^1$ mapping, its $B$-subdifferential can be computed by simply computing the Jacobians of its smooth pieces. Specifically we have
$$\partial_B(\prox_{\gamma g})(x)=\begin{cases}\left\{I-\gamma/\|x\|_2\left(I-ww'\right)\right\},&\textrm{ if } \|x\|_2>\gamma,\\
\{0\},&\textrm{ if } \|x\|_2<\gamma,\\
\left\{I-\gamma/\|x\|_2\left(I-ww'\right),0\right\},&\textrm{ otherwise}.\end{cases}$$
where $w=x/\|x\|_2$.
\subsubsection{$\ell_1$ norm}\label{ex:EllOne}
The proximal mapping of $g(x)=\|x\|_1$ is the well known soft-thresholding operator
$$(\prox_{\gamma g}(x))_i=(\sign(x_i)(|x_i|-\gamma)_+)_i,\quad i\in[n].$$
Function $g$ is separable, therefore according to Proposition~\ref{prop:Separ} every element of $\partial_B(\prox_{\gamma g})$ is a diagonal matrix. The explicit form of the elements of $\partial_B(\prox_{\gamma g})$ is as follows.
Let $\alpha=\{i\ |\ |x_i|>\gamma\}$, $\beta=\{i\ |\ |x_i|=\gamma\}$, $\delta=\{i\ |\ |x_i|<\gamma\}$. Then
$P\in\partial_B(\prox_{\gamma g})(x)$ if and only if $P$ is diagonal with elements
$$P_{ii}=\begin{cases}
1,&\textrm{ if } i\in\alpha,\\
\in\{0,1\},&\textrm{ if } i\in\beta,\\
0,&\textrm{ if } i\in\delta.
\end{cases}$$
We could also arrive to the same conclusion by applying Proposition~\ref{prop:MorDec} to the function of Section~\ref{ex:projBox} with $u=-\ell=\mathbf{1}_n$, since the $\ell_1$ norm is the conjugate of the indicator of the $\ell_\infty$ -norm ball.
\subsubsection{Sum of norms} If $g(x)=\sum_{s\in\mathcal{S}}\|x_s\|_2$, where $\mathcal{S}$ is a partition of $[n]$, then
$$(\prox_{\gamma g}(x))_s=\left(1-\frac{\gamma}{\|x_s\|_2}\right)_+x_s,$$
for all $s\in\mathcal{S}$. Any $P\in\partial_B(\prox_{\gamma g})(x)$ is block diagonal with the $s$-th block equal to $I-\gamma/\|x_s\|_2\left(I-(1/\|x_s\|_2^2)x_sx_s'\right)$, if $\|x_s\|_2>\gamma$, $I$ if $\|x_s\|_2<\gamma$ and any of these two matrices if $\|x_s\|_2=\gamma$.
\subsection{Support function}\label{ex:SuppFun} Since $\sigma_C(x) = \sup_{y\in C}x'y$ is the conjugate of the indicator $\delta_C$, one can use Proposition~\ref{prop:MorDec} to find that
$$\partial_B(\prox_{\gamma g})(x) = \left\{P = I-Q : Q\in\partial_B(\Pi_C)(x/\gamma)\right\}.$$
Depending on the specific set $C$ (see Section~\ref{ex:IndFun}) one obtains the appropriate subdifferential. A particular example is the following.
\subsection{Pointwise maximum}
Function $g(x)=\max\{x_1,\ldots,x_n\}$ is conjugate to the indicator of the unit simplex already analyzed in Section~\ref{ex:projSimplex}. Applying Proposition~\ref{prop:MorDec} we obtain
$$\partial_B(\prox_{\gamma g})(x)=\{P=I-Q\ |\ Q\in\partial_B(\Pi_D)(x/\gamma)\}$$
Then $\Pi_D(x/\gamma)=(x/\gamma-\lambda\mathbf{1})_+$ where $\lambda$ solves $\mathbf{1}'(x/\gamma-\lambda\mathbf{1})_+=1$.
Let $J=\{i\in[n]\ |\ (\Pi_D(x/\gamma))_i=0\}$, $n_J=|J|$ and $J_c=[n]\setminus J$. It follows that an element of $\partial_B(\prox_{\gamma g})(x)$  is block-diagonal (after a reordering of variables) with
$$P_{ij}=\begin{cases}1,&\textrm{ if } i,j\in J\\
1+1/(n-n_J),&\textrm{ if } i\neq j, i,j\in J_c,\\
1/(n-n_J),&\textrm{ if } i= j, i,j\in J_c.\end{cases}$$
\subsection{Spectral functions}
 For any symmetric $n$ by $n$ matrix $X$, the eigenvalue function $\lambda:\Ss^n\to\Re^n$ returns the vector of its eigenvalues in nonincreasing order. 
 Now consider function $G:\Ss^n\to\bar{\Re}$
\begin{equation}\label{eq:SpecFun}
G(X)=h(\lambda(X)),\quad X\in\Ss^n,
\end{equation} 
where  $h:\Re^n\to\bar{\Re}$ is proper, closed, convex and symmetric, \ie invariant under coordinate permutations.
Functions of this form are called \emph{spectral functions}\cite{lewis1996convex}. Being a spectral function, $G$ inherits most of the properties of $h$\cite{lewis1996derivatives, lewis2001twice}. In particular, its proximal mapping is simply\cite[Sec.~6.7]{parikh2013proximal}
$$\prox_{\gamma G}(X)=Q\diag(\prox_{\gamma h}(\lambda(X)))Q',$$
where $X=Q\diag(\lambda(X))Q'$ is the spectral decomposition of $X$ ($Q$ is an orthogonal matrix).
Next, we further assume that
\begin{equation}\label{eq:SymSep}
h(x)=g(x_1)+\cdots+g(x_N),
\end{equation}
where $g:\Re\to\bar{\Re}$.
Since $h$ is also separable we have that 
$$\prox_{\gamma h}(x)=(\prox_{\gamma g}(x_1),\ldots,\prox_{\gamma g}(x_N)),$$ 
therefore the proximal mapping of $G$ can be expressed as
\begin{equation}\label{eq:ProxSpec}
\prox_{\gamma G}(X)=Q\diag(\prox_{\gamma g}(\lambda_1(X)),\ldots,\prox_{\gamma g}(\lambda_n(X)))Q'.
\end{equation}
Functions of this form are called \emph{symmetric matrix-valued functions}
~\cite[Chap. V]{bhatia1997matrix},~\cite[Sec.~6.2]{horn1991topics}.
Now we can use the theory of nonsmooth symmetric matrix-valued functions developed in~\cite{chen2003analysis} to
analyze differentiability properties of $\prox_{\gamma G}$. In particular $\prox_{\gamma G}$ is (strongly)
semismooth at $X$ if and only if $\prox_{\gamma g}$ is (strongly) semismooth at the eigenvalues of X~\cite[Prop.~4.10]{chen2003analysis}.
Moreover, for any $X\in\Ss^n$ and $P\in\partial_B(\prox_{\gamma G})(X)$ we have~\cite[Lem.~4.7]{chen2003analysis} 
\begin{equation}\label{eq:JacSpec}
P(S)=Q(\Omega\circ(Q'SQ))Q',\ \forall S\in\Ss^n,
\end{equation}
where $\circ$ denotes the Hadamard product and the matrix $\Omega\in\Re^{n\times n}$ is defined by
\begin{equation}\label{eq:GammaJac}
\Omega_{ij}=\begin{cases}
\frac{\prox_{\gamma g}(\lambda_i)-\prox_{\gamma g}(\lambda_j)}{\lambda_i-\lambda_j},&\textrm{ if } \lambda_i\neq\lambda_j,\\
\in\partial(\prox_{\gamma g})(\lambda_i),&\textrm{ if } \lambda_i=\lambda_j.
\end{cases}
\end{equation}
\subsubsection{Indicator of the positive semidefinite cone}
The indicator of $\Ss_+^n$ can be expressed as in~\eqref{eq:SpecFun} with $h$ given by~\eqref{eq:SymSep} and $g=\delta_{\Re_+}$. Then $\prox_{\gamma g}(x)=\Pi_{\Re_+}(x)=(x)_+$ and according to~\eqref{eq:ProxSpec} we have
$$\Pi_{\Ss^n_+}(X)=Q\diag((\lambda_1)_+,\ldots,(\lambda_n)_+)Q'.$$
Let
$\alpha=\{i\ |\ \lambda_i>0\}$ and $\bar{\alpha}=[n]\setminus\alpha$.  An element of $\partial_B\Pi_{\mathbb{S}_+^n}(X)$ is given by~\eqref{eq:JacSpec} with
$$\Omega =\begin{bmatrix} \Omega_{\alpha\alpha}& k_{\alpha\bar{\alpha}}\\  k_{\alpha\bar{\alpha}}'&0 \end{bmatrix},$$
where $\Omega_{\alpha\alpha}$ is a matrix of ones and $ k_{ij}=\frac{\lambda_i}{\lambda_i-\lambda_j},\ i\in\alpha,\ j\in\bar{\alpha}$. In fact we have $P(S)=H+H'$~\cite[Sec.~4]{zhao2010newton} where
$$H=Q_\alpha\left(\tfrac{1}{2}(UQ_\alpha)Q_\alpha'+( k_{\alpha\bar{\alpha}}\circ(UQ_{\bar{\alpha}}))Q_{\bar{\alpha}}'\right)$$
and $U=Q_{\alpha}'S$. Therefore we can form $P(S)$ in at most $8|\alpha|n^2$ flops. When $|\alpha|>|\bar{\alpha}|$, we can alternatively express $P(S)$ as $S-Q'((E-\Omega)\circ(Q'SQ))Q'$, where $E$ is a matrix of all ones and compute it in $8|\bar{\alpha}|n^2$ flops.

\subsection{Orthogonally invariant functions}
A function $G:\Re^{m\times n}\to\bar{\Re}$ is called \emph{orthogonally invariant} if
$$G(UXV')=G(X),$$
for all $X\in\Re^{m\times n}$ and all orthogonal matrices $U\in\Re^{m\times m}$, $V\in\Re^{n\times n}$. When the elements of $X$ are allowed to be complex numbers then functions of this form are called \emph{unitarily invariant}~\cite{lewis1995convex}.
A function $h:\Re^{q}\to\bar{\Re}$ is \emph{absolutely symmetric} if $h(Qx)=h(x)$ for all $x\in\Re^p$ and any generalized permutation matrix $Q$, \ie a matrix $Q\in\Re^{q\times q}$ that has exactly one nonzero entry in each row and each column, that entry being $\pm 1$~\cite{lewis1995convex}.
There is a one-to-one correspondence between orthogonally invariant functions on $\Re^{m\times n}$ and absolutely symmetric functions on $\Re^q$. Specifically if $G$ is orthogonally invariant then 
$$G(X)=h(\sigma(X)),$$
for the absolutely symmetric function $h(x)=G(\diag(x))$. Here for $X\in\Re^{m\times n}$, the spectral function
$\sigma:\Re^{m\times n}\to\Re^q$, $q=\min\{m,n\}$ returns the vector of its singular values in nonincreasing order.
Conversely, if $h$ is absolutely symmetric then $G(X)= h(\sigma(X))$ is orthogonally invariant.
Therefore, convex-analytic and generalized differentiability properties of orthogonally invariant functions can be
easily derived  from those of the corresponding absolutely symmetric functions~\cite{lewis1995convex}.
For example, assuming for simplicity that $m\leq n$, the proximal mapping of $G$ is given by
(see e.g. \cite[Sec.~6.7]{parikh2013proximal})
$$\prox_{\gamma G}(X)=U\diag(\prox_{\gamma h}(\sigma(X)))V_1',$$
where $X=U\begin{bmatrix}\diag(\sigma(X)),&0\end{bmatrix}\begin{bmatrix}V_1,& V_2\end{bmatrix}'$ is the singular value decomposition of $X$.
If we further assume that $h$ is separable as in~\eqref{eq:SymSep} then
\begin{equation}\label{eq:ProxSpec2}
\prox_{\gamma G}(X)=U\Sigma_g(X)V_1',
\end{equation}
where $\Sigma_g(X)=\diag(\prox_{\gamma g}(\sigma_1(X)),\ldots,\prox_{\gamma g}(\sigma_n(X)))$.
Functions of this form are called \emph{nonsymmetric matrix-valued functions}. We also assume that $g$ is a non-negative function such that $g(0)=0$.
This implies that $\prox_{\gamma g}(0)=0$ and guarantees that the nonsymmetric matrix-valued function~\eqref{eq:ProxSpec2} is well-defined~\cite[Prop.~2.1.1]{yang2009study}. Now we can use the results of~\cite[Ch. 2]{yang2009study} to  draw conclusions about generalized differentiability properties of $\prox_{\gamma G}$.
For example, through~\cite[Th. 2.27]{yang2009study} we have that 
$\prox_{\gamma G}$ is continuously differentiable at $X$ if and only if $\prox_{\gamma g}$ is continuously differentiable at the singular values of $X$. Furthermore, $\prox_{\gamma G}$ is (strongly) semismooth at $X$ if $\prox_{\gamma g}$ is (strongly) semismooth at the singular values of  $X$ \cite[Th. 2.3.11]{yang2009study}. 

For any $X\in\Re^{m\times n}$ the generalized Jacobian $\partial_B(\prox_{\gamma G}) (X)$ is well defined and nonempty and
any $P\in\partial_B(\prox_{\gamma G})(X)$ acts on $H\in\Re^{m\times n}$ as \cite[Prop.~2.3.7]{yang2009study}
\begin{equation}\label{eq:JacNS}
P(H)=U\begin{bmatrix}\left(\Omega_{1}\circ\left(\frac{H_1+H_1'}{2}\right)+\Omega_{2}\circ\left(\frac{H_1-H_1'}{2}\right)\right),&(\Omega_{3}\circ H_2)\end{bmatrix}\begin{bmatrix}V_1,&V_2\end{bmatrix}'
\end{equation}
where $H_1=U'HV_1\in\Re^{m\times m}$, $H_2=U'HV_2\in\Re^{m\times(n-m)}$ 
and $\Omega_{1}\in\Re^{m\times m }$, $\Omega_{2}\in\Re^{m\times m }$, $\Omega_{3}\in\Re^{m\times (n-m) }$ are given by 
\begin{align*}
(\Omega_{1})_{ij}&=\begin{cases}
\frac{\prox_{\gamma g}(\sigma_i)-\prox_{\gamma g}(\sigma_j)}{\sigma_i-\sigma_j},&\textrm{ if } \sigma_i\neq\sigma_j,\\
\in\partial \prox_{\gamma g}(\sigma_i),&\textrm{ if }\sigma_i=\sigma_j,
\end{cases}\\
(\Omega_{2})_{ij}&=\begin{cases}
\frac{\prox_{\gamma g}(\sigma_i)-\prox_{\gamma g}(-\sigma_j)}{\sigma_i+\sigma_j},&\textrm{ if } \sigma_i\neq -\sigma_j,\\
\in\partial \prox_{\gamma g}(0),&\textrm{ if }\sigma_i=\sigma_j=0,
\end{cases}\\
(\Omega_{3})_{ij}&=\begin{cases}
\frac{\prox_{\gamma g}(\sigma_i)}{\sigma_i},&\textrm{ if } \sigma_i\neq 0,\\
\in\partial \prox_{\gamma g}(0),&\textrm{ if }\sigma_i=0.
\end{cases}
\end{align*}

\subsubsection{Nuclear norm} For an $m$ by $n$ matrix $X$ the nuclear norm, $G(X)=\|X\|_{*}$, is the sum of its singular values, \ie $G(X)=\sum_{i=1}^m\sigma_i(X)$
(we are again assuming, for simplicity, that $m\leq n$). The nuclear norm serves as a convex surrogate for the rank of a matrix. 
It has found many applications in systems and control theory, including system identification and model reduction~\cite{fazel2001rank,fazel2002matrix,fazel2004rank, liu2009interior,recht2010guaranteed}.
Other fields of application include \emph{matrix completion problems} arising in machine learning~\cite{srebro2004learning,rennie2005fast}
and computer vision~\cite{tomasi1992shape,morita1997sequential}, and \emph{nonnegative matrix factorization problems} arising in data mining~\cite{elden2007matrix}.

The nuclear norm can be expressed as  $G(X)=h(\sigma(X))$, where $h(x)=\|x\|_1$. 
Apparently, $h$ is absolutely symmetric and separable. Specifically, it takes the form~\eqref{eq:SymSep} with $g=|\cdot|$, for which $0\in\dom g$ and $0\in\partial g(0)$.
The proximal mapping of the absolute value is the soft-thresholding operator. 
In fact, since the case of interest here is $x\geq 0$ (because $\sigma_i(X)\geq 0$), we have $\prox_{\gamma g}(x)=(x-\gamma)_+$. Consequently, the proximal mapping of
$\|X\|_*$ is given by~\eqref{eq:ProxSpec2} with 
$$\Sigma_g(X)=\diag((\sigma_1(X)-\gamma)_+,\ldots,(\sigma_m(X)-\gamma)_+).$$
For $x\in\Re_+$ we have that 
\begin{equation}\label{eq:subSoft}
\partial(\prox_{\gamma g})(x)=\begin{cases}
0,&\textrm{ if } 0\leq x<\gamma,\\
[0,1],&\textrm{ if } x=\gamma,\\
1,&\textrm{ if } x>\gamma.
\end{cases}
\end{equation}
Let $\alpha=\{i\ |\ \sigma_i(X)>\gamma\}$, $\beta=\{i\ |\ \sigma_i(X)=\gamma\}$ and $\delta=\{i\ |\ \sigma_i(X)<\gamma\}$.
Taking into account~\eqref{eq:subSoft}, an element $P$ of the $B$-subdifferential $\partial_B(\prox_{\gamma G})(X)$ satisfies~\eqref{eq:JacNS} with
\begin{align*}
\Omega_1&=
    \begin{bmatrix}\omega_{\alpha\alpha}^1&\omega_{\alpha\beta}^1&\omega_{\alpha\delta}^1\\
    (\omega_{\alpha\beta}^1)'&\omega_{\beta\beta}^1&0\\
    (\omega_{\alpha\delta}^1)'&0&0\end{bmatrix},    
\quad&
    \begin{array}{ll}\omega^1_{ij}=1, &i\in\alpha, j\in\alpha\cup\beta,\\
    \omega^1_{ij}=\frac{\sigma_i(X)-\gamma}{\sigma_i(X)-\sigma_j(X)},& i\in\alpha, j\in\delta,\\
    \omega_{ij}^1=\omega_{ji}^1=[0,1],&i,j\in\beta\end{array}\\
\Omega_2&=
    \begin{bmatrix}\omega_{\alpha\alpha}^2&\omega_{\alpha\beta}^2&\omega_{\alpha\delta}^2\\
    (\omega_{\alpha\beta}^2)'&0&0\\
    (\omega_{\alpha\delta}^2)'&0&0\end{bmatrix},
\quad&
    \begin{array}{ll}\\
    \omega^2_{ij}=\frac{(\sigma_i(X)-\gamma)_++(\sigma_j(X)-\gamma)_+}{\sigma_i(X)+\sigma_j(X)},& i\in\alpha, j\in [m],\\
    \\\end{array}\\
\Omega_3&=
    \begin{bmatrix}\omega_{\alpha [n-m]}^3\\
    0\end{bmatrix},
\quad&
    \begin{array}{ll}\\
    \omega^3_{ij}=\frac{\sigma_i(X)-\gamma}{\sigma_i(X)},& i\in\alpha, j\in [n-m].\\
    \\
\end{array}
\end{align*}

\section{Simulations}\label{sec:Simulations}
This section is devoted to the application of Algorithms~\ref{al:PNM} and~\ref{al:PGNM} to some practical problems. Based on the results obtained in Section~\ref{sec:Examples},
we discuss the Newton system for each of the examples, and compare the proposed approach against other algorithms on the basis of numerical results obtained with \matlab.


\subsection{Box constrained QPs}
A quadratic program with box constraints
can be reformulated in the form \eqref{eq:GenProb} by adding to the cost the indicator of the feasible set, namely $\delta_{[l,u]}$. Then
$$ f(x) = \frac{1}{2}x'Qx + q'x,\quad g(x) = \delta_{[l,u]}(x). $$
The B-subdifferential, in this case, is composed of diagonal matrices, with diagonal
elements in $\{0,1\}$, cf. Section~\ref{ex:projBox}. More precisely, in Algorithm~\ref{al:PNM}, we can split
variable indices in the two sets
\begin{align*}
\alpha &= \left\{i\ \left.\right|\ l_i < \left[x-\gamma\nabla f(x)\right]_i < u_i\right\},\\
\bar\alpha &= \left\{1,\ldots, n\right\}\setminus \alpha,
\end{align*}
and choose $P=\diag(p_1,\ldots,p_n)$, with $p_i=1$ if $i\in\alpha$ and $p_i=0$ otherwise.
Then the Newton system~\eqref{eq:RegNewtSys} reduces the triangular form
$$ \begin{bmatrix} I_{|\bar\alpha|} & \\ \gamma Q_{\alpha\bar\alpha} & \gamma Q_{\alpha\alpha} \end{bmatrix} d^k = P_{\gamma}(x^k) - x^k.$$
This can be solved by forward substitution, where only the $|\alpha|$-by-$|\alpha|$ block is solved via CG.
We tested the proposed algorithms against the commercial QP solver \gurobi, \matlab's built-in ``quadprog'' solver,
the accelerated forward-backward splitting~\cite{nesterov2007gradient}
(with constant stepsize) and the alternating directions method of multipliers (ADMM)~\cite{boyd2011distributed}. The latter was both implemented using a direct solver, which requires
the initial computation of the Cholesky factor of $Q$, and the conjugate gradient method. Random problems were generated with chosen size, density and condition number,
as explained in~\cite{gonzaga2013optimal}. Figures~\ref{fig:QPbox_size_cond}-\ref{fig:QP_time_vs_obj} show the results obtained:
the proposed algorithms are generally faster then the others, and also appear to scale
good with respect to problem size and condition number.

\begin{figure}[tp!]
\center
\begin{tabular}{cc}
\subfloat[Problem size]{\includegraphics[width=0.49\textwidth]{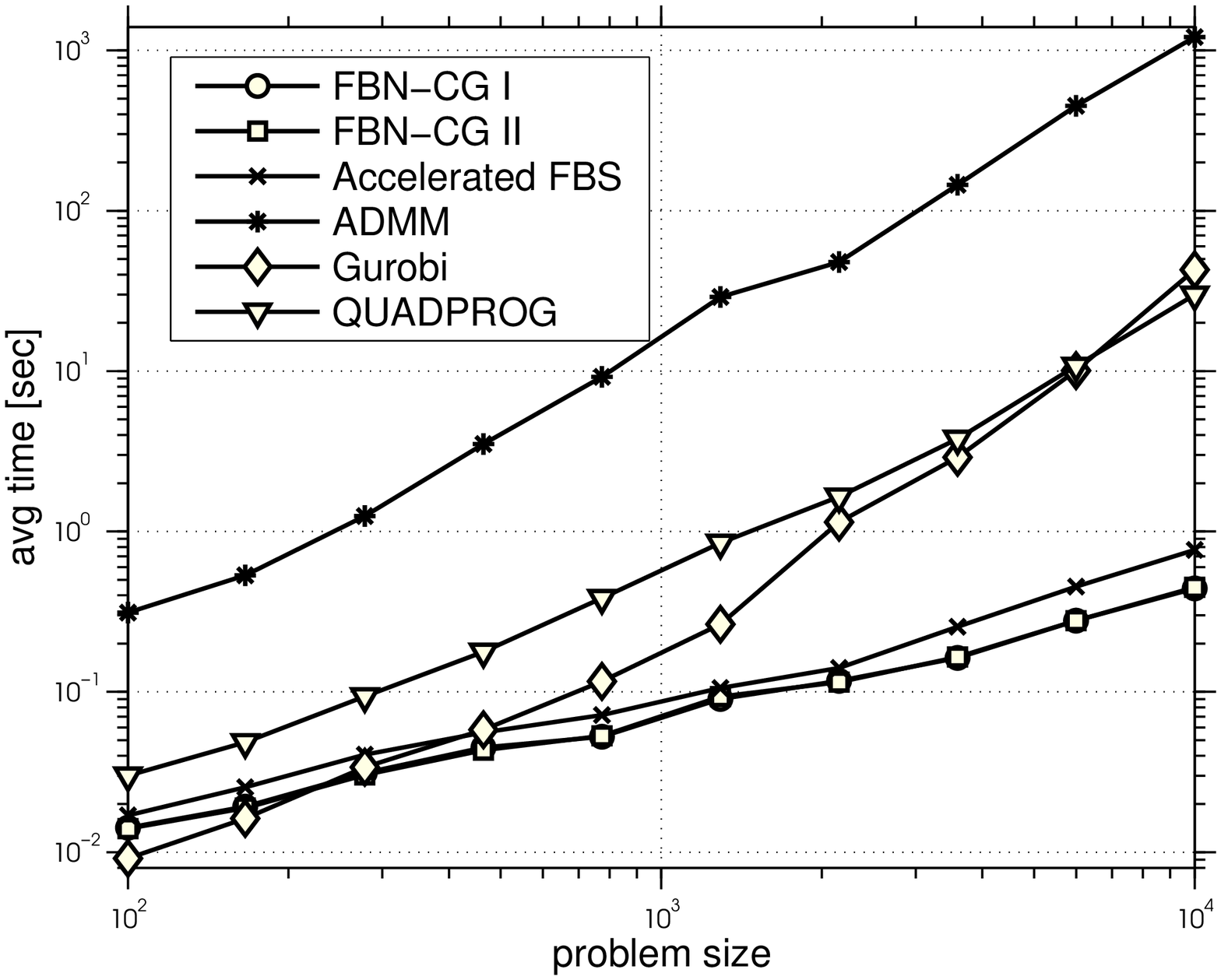}} &
\subfloat[Condition number]{\includegraphics[width=0.49\textwidth]{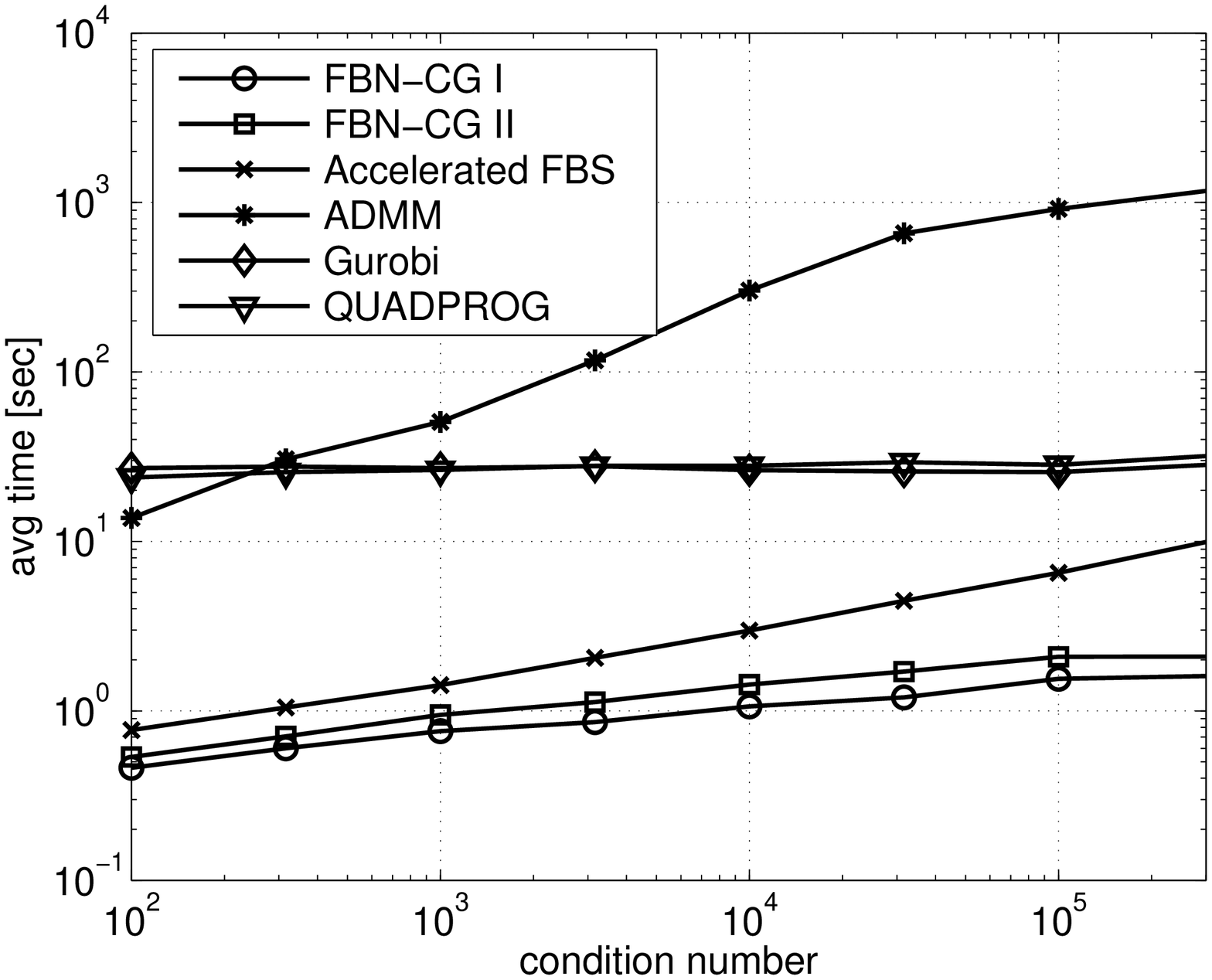}}
\end{tabular}
\caption{Box constrained QPs. Average running times over a sample of $20$ random instances, with increasing problem size and condition number.}
\label{fig:QPbox_size_cond}
\end{figure}


\subsection{General QPs}
If we consider the more general quadratic programming problem with constraint $l\leq Ax \leq u$,
$A\in\Re^{m\times n}$, then the projection onto the feasible set is not
explicitly computable like in the previous example.
Formulating the Fenchel dual, and letting $w$ be the dual variable, one can tackle the composite problem with
$$ f(w) = \frac{1}{2}(A'w+q)'Q^{-1}(A'w+q),\qquad g(w) = \sigma_{[l,u]}(w).$$
Also in this case $\prox_{\gamma g}(w) = w-\Pi_{[\gamma l,\gamma u]}(w)$
has its B-subdifferential composed of binary diagonal matrices, cf. Section~\ref{ex:SuppFun}:
\begin{align*}
\bar\alpha &= \left\{i\ \left.\right|\ \gamma l_i \leq \left[x-\gamma\nabla f(x)\right]_i \leq \gamma u_i\right\},\\
\alpha &= \left\{1,\ldots, n\right\}\setminus \alpha.
\end{align*}
Choosing $P=\diag(p_1,\ldots,p_n)$, with $p_i=1$ if $i\in\alpha$ and $p_i=0$ otherwise,
just like in the previous case system~\eqref{eq:RegNewtSys} is block-triangular:
$$ \begin{bmatrix} I_{|\bar\alpha|} & \\ \gamma A_{\alpha}Q^{-1}A_{\bar\alpha}' & \gamma A_{\alpha}Q^{-1}A_{\alpha}' \end{bmatrix} d = P_{\gamma}(w) - w. $$
Here subscripts denote \emph{row} subsets. The latter is solved by forward
substitution, and the $|\alpha|$-by-$|\alpha|$ block is solved via CG.
Note that all the products with $Q^{-1}$ are merely formal, and require a
previous computation of the Cholesky factor of $Q$. Figure~\ref{fig:QP_time_vs_obj} compares Algorithm~\ref{al:PNM} and \ref{al:PGNM} to the accelerated
version of FBS~\cite{nesterov2007gradient} and to ADMM~\cite{boyd2011distributed}, in terms of objective value decrease.

\begin{figure}[tp!]
\center
\begin{tabular}{cc}
\subfloat[Box constrained QP, $n=1500$ and $\kappa = 10^4$.]{\includegraphics[width=0.49\textwidth]{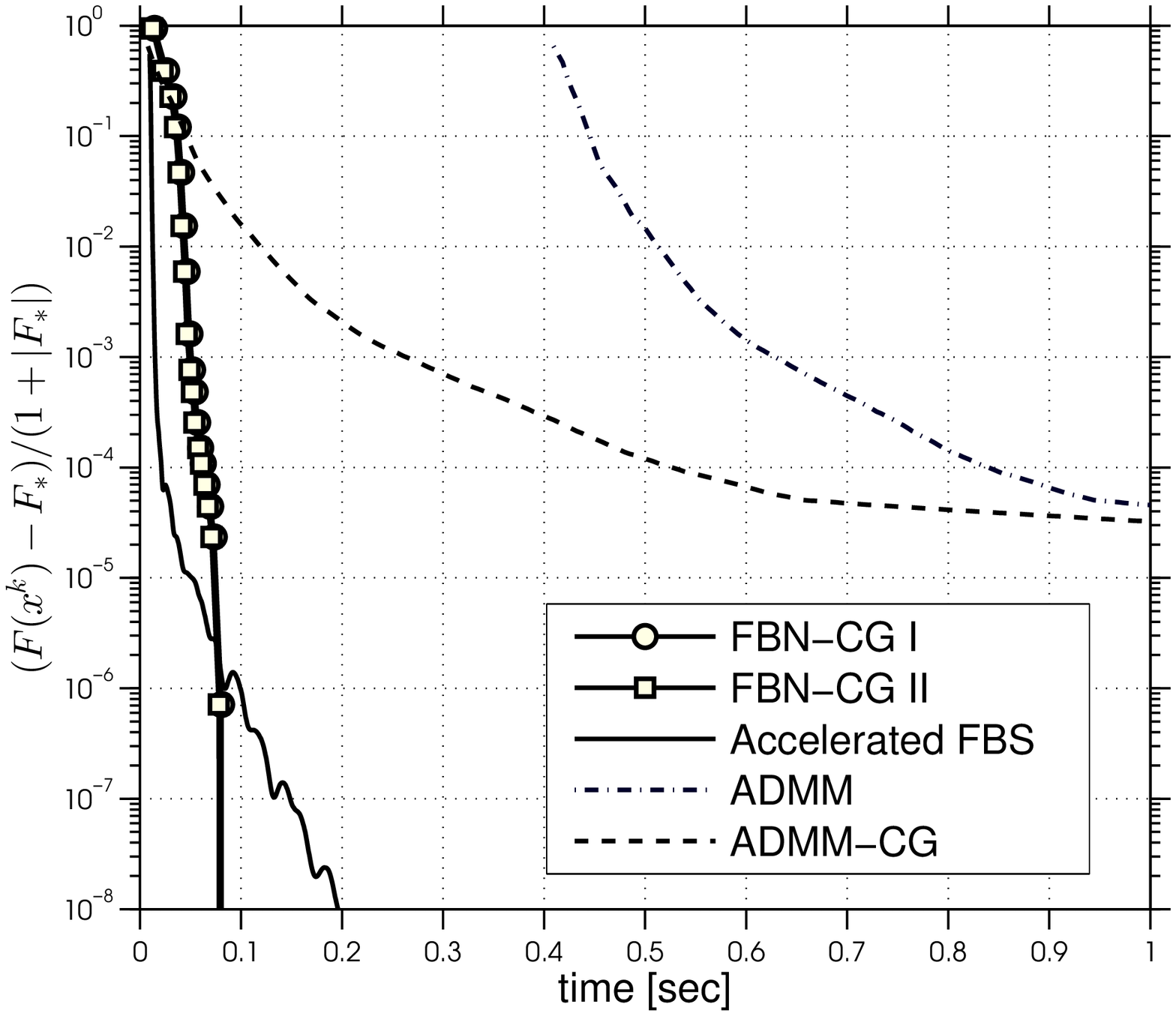}} &
\subfloat[General QP, $n=1000$, $m=2000$ and $\kappa = 10^3$.]{\includegraphics[width=0.49\textwidth]{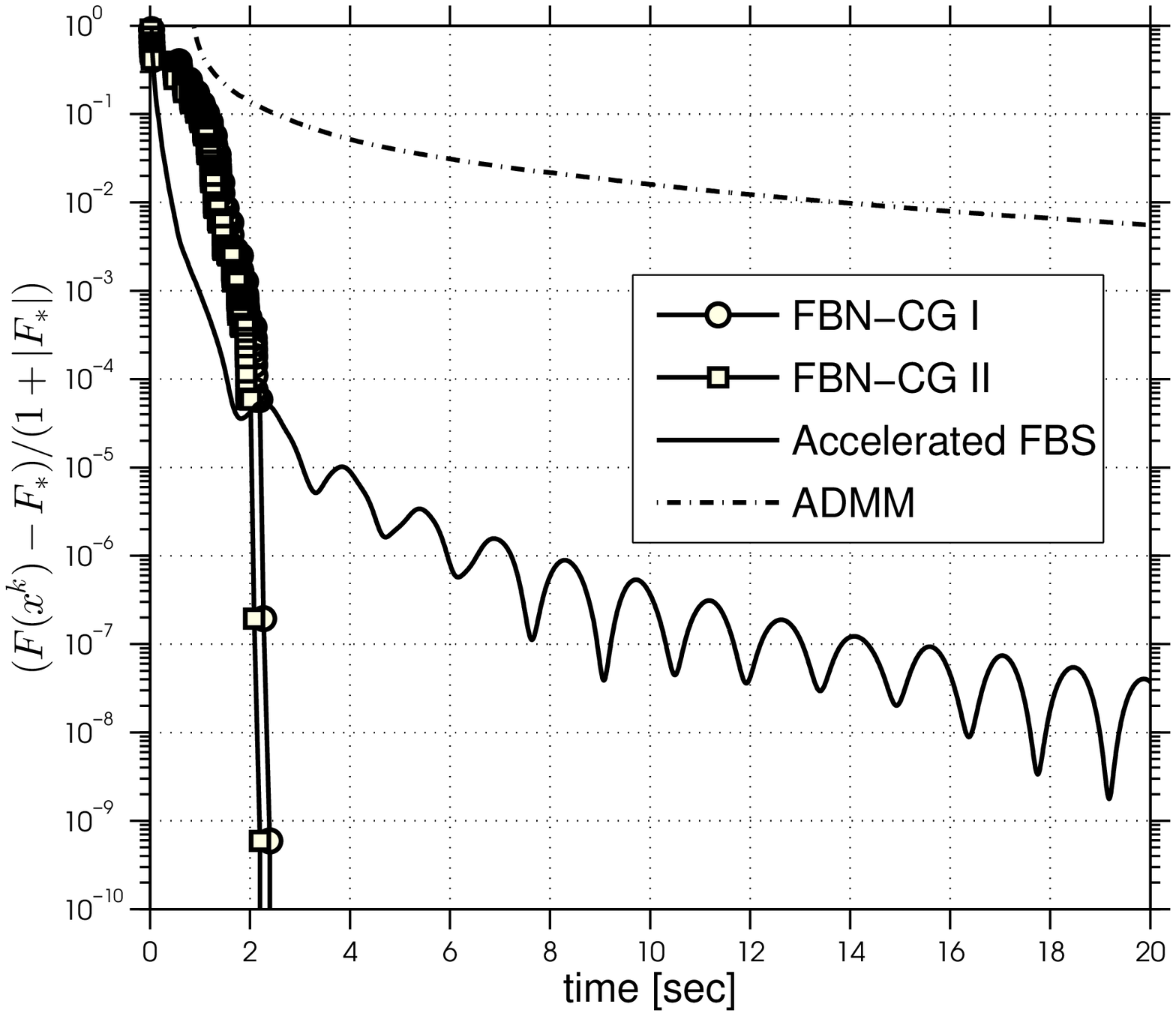}}
\end{tabular}
\caption{QPs. Comparison of the methods applied to a box constrained QP (primal) and to a general QP (dual).}
\label{fig:QP_time_vs_obj}
\end{figure}


\subsection{\texorpdfstring{$\ell_1$}{l1}-regularized least squares}
This is a classical problem arising in many fields like statistics, machine learning, signal and image processing.
The purpose is to find a sparse solution to an underdetermined linear system.
We have
$$ f(x) = \frac{1}{2}\|Ax-b\|_2^2,\qquad g(x) = \lambda\|x\|_1,$$
where $A\in\Re^{m\times n}$ with $m < n$. The $\ell_1$-regularization term is known to promote sparsity in the solution vector $x^*$.
As we mentioned in Section~\ref{ex:EllOne}, the proximal mapping of the $\ell_1$ norm is the soft-thresholding operator, whose generalized Jacobian is diagonal.
Specifically, if
\begin{align*}
\alpha &= \left\{i\ \left.\right|\ \left|\left[x-\gamma\nabla f(x)\right]_i\right| > \gamma\lambda\right\},\\
\bar\alpha &= \left\{1,\ldots, n\right\}\setminus \alpha,
\end{align*}
then $P=\diag(p_1,\ldots,p_n)$, with $p_i=1$ if $i\in\alpha$ and $p_i=0$ otherwise, is an element of $\partial_B(\prox_{\gamma g})(x-\gamma\nabla f(x))$.
The simplified system~\eqref{eq:RegNewtSys} reduces then to
\begin{equation}
\begin{bmatrix}
I_{|\bar\alpha|} & \\
\gamma A_{\alpha}'A_{\bar\alpha} & \gamma A_{\alpha}'A_{\alpha}
\end{bmatrix}d = P_{\gamma}(x)-x.
\end{equation}
Here subscripts denote \emph{column} subsets.
The dimension of the problem to solve at each iteration is then $|\alpha|$: the smaller this set is, the cheaper
the computation of the Newton direction is.
Noting that at, any given $x$, larger values of $\lambda$ allow for smaller size of $\alpha$, and that decreasing $\lambda$ decreases the objective value,
we can set up a simple continuation scheme in order to keep the size of $\alpha$ small:
starting from a relatively large value of $\lambda = \lambda_{\mbox{\tiny max}} > \lambda_0$, we decrease it every time a certain criterion is met until $\lambda = \lambda_0$,
using the solution of one step as to warm-start the next one. Specifically, we set $\lambda_{\mbox{\tiny max}} = \|A'b\|_{\infty}$,
which is the threshold above which the null solution is optimal. For an in-depth analysis of such continuation techniques on this type of problems, see \cite{xiao2013proximal}.
We compared our method to SpaRSA~\cite{wright2009sparse},
YALL1~\cite{yang2011alternating} and l1\_ls~\cite{koh2007interior}.
The algorithms were tested against the datasets available at
\url{wwwopt.mathematik.tu-darmstadt.de/spear}~\cite{lorenz2013constructing}.
These include datasets with different sizes and dynamic ranges of the solution.
In each test we obtained a reference solution by running the method extensively, with
a very small tolerance as stopping criterion. Then we set all the algorithms to stop
as soon as the primal objective value reached a threshold at a relative distance
$\epsilon_r = 10^{-8}$ from the reference solution.
Figure~\ref{fig:l1ls_1} reports the performance profiles \cite{dolan2012benchmarking}
of the algorithms considered on the aforementioned problem set.
A point $(r,f)$ on a line indicates that the correspondent algorithm had a
performance ratio\footnote{An algorithm has a performance ratio $r$, with respect
to a problem, if its running time is $r$ times the running time of the top
performing algorithm among the ones considered.} at most $r$ in a fraction $f$ of problems.
It appears that the forward-backward Newton-CG method is very stable compared to the other algorithms considered.
The benefits of the continuation scheme are evident from Figure \ref{fig:l1ls_3}, where the size of the linear system solved by FBN-CG at every iteration is shown.

\begin{figure}[tp!]
\center
\includegraphics[width=0.7\textwidth]{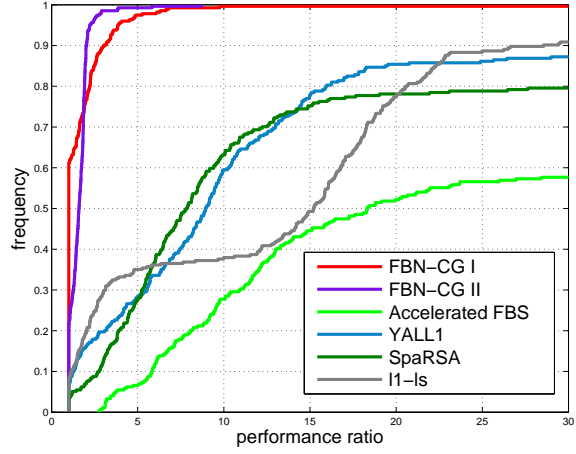}
\caption{$\ell_1$-regularized least squares. Performance profiles of the algorithms on the SPEAR test set with $\lambda_0=10^{-3}\lambda_{\mbox{\tiny max}}$.
The FBN-CG methods considered perform continuation on $\lambda$.}
\label{fig:l1ls_1}
\end{figure}

\begin{figure}[tp!]
\center
\begin{tabular}{cc}
\subfloat[\emph{spear\_inst\_1}]{\includegraphics[width=0.49\textwidth]{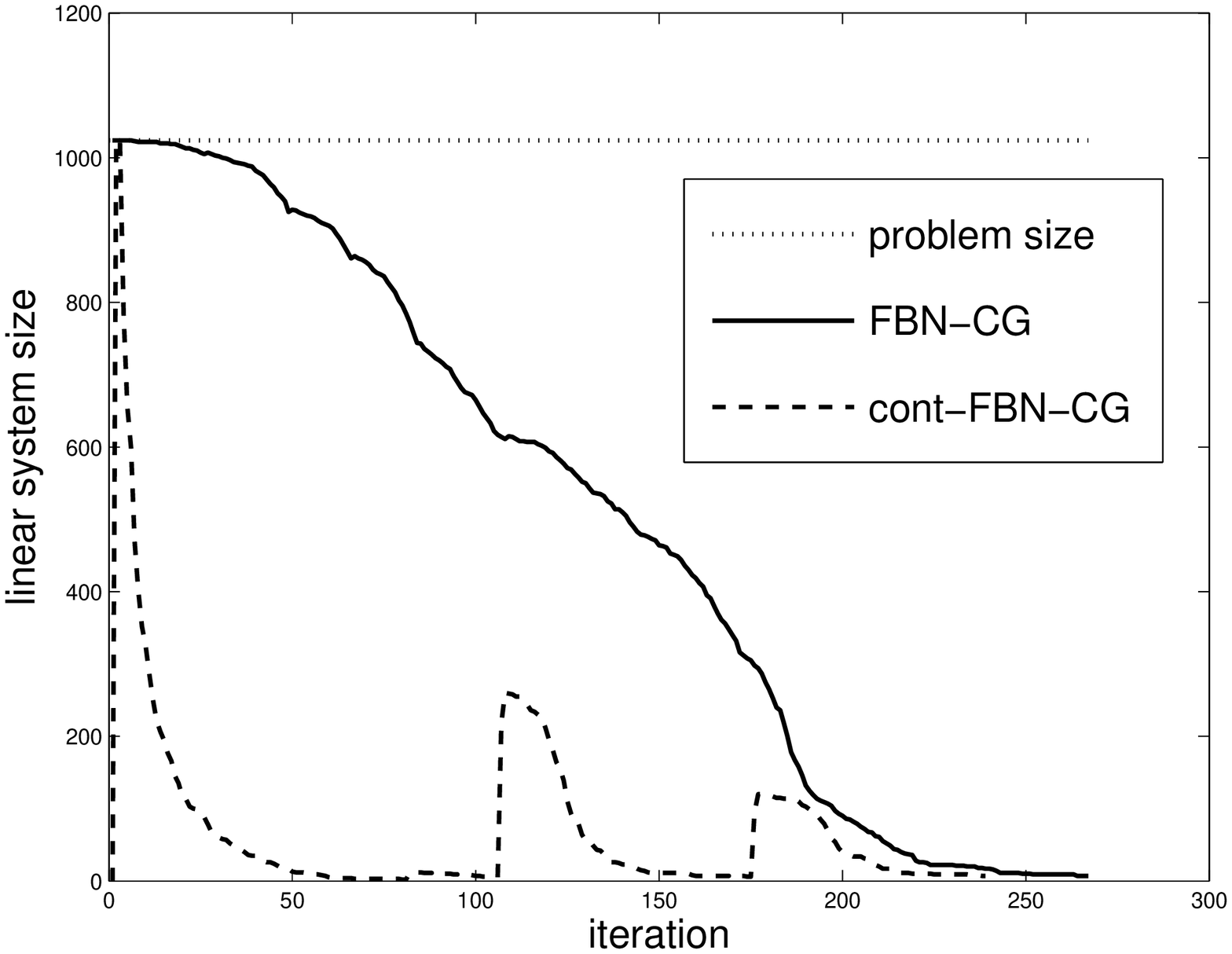}} &
\subfloat[\emph{spear\_inst\_91}]{\includegraphics[width=0.49\textwidth]{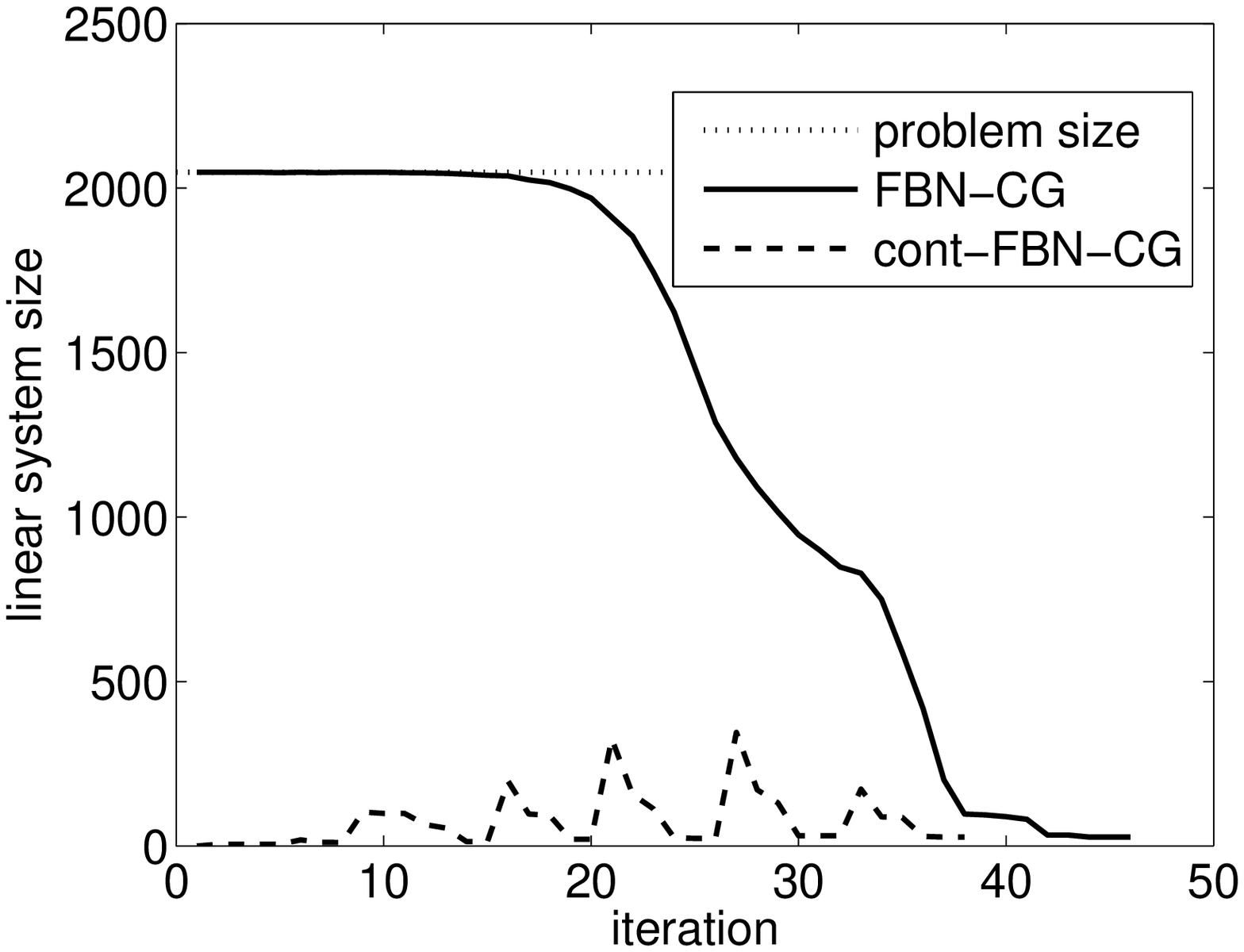}} \\
\subfloat[\emph{spear\_inst\_131}]{\includegraphics[width=0.49\textwidth]{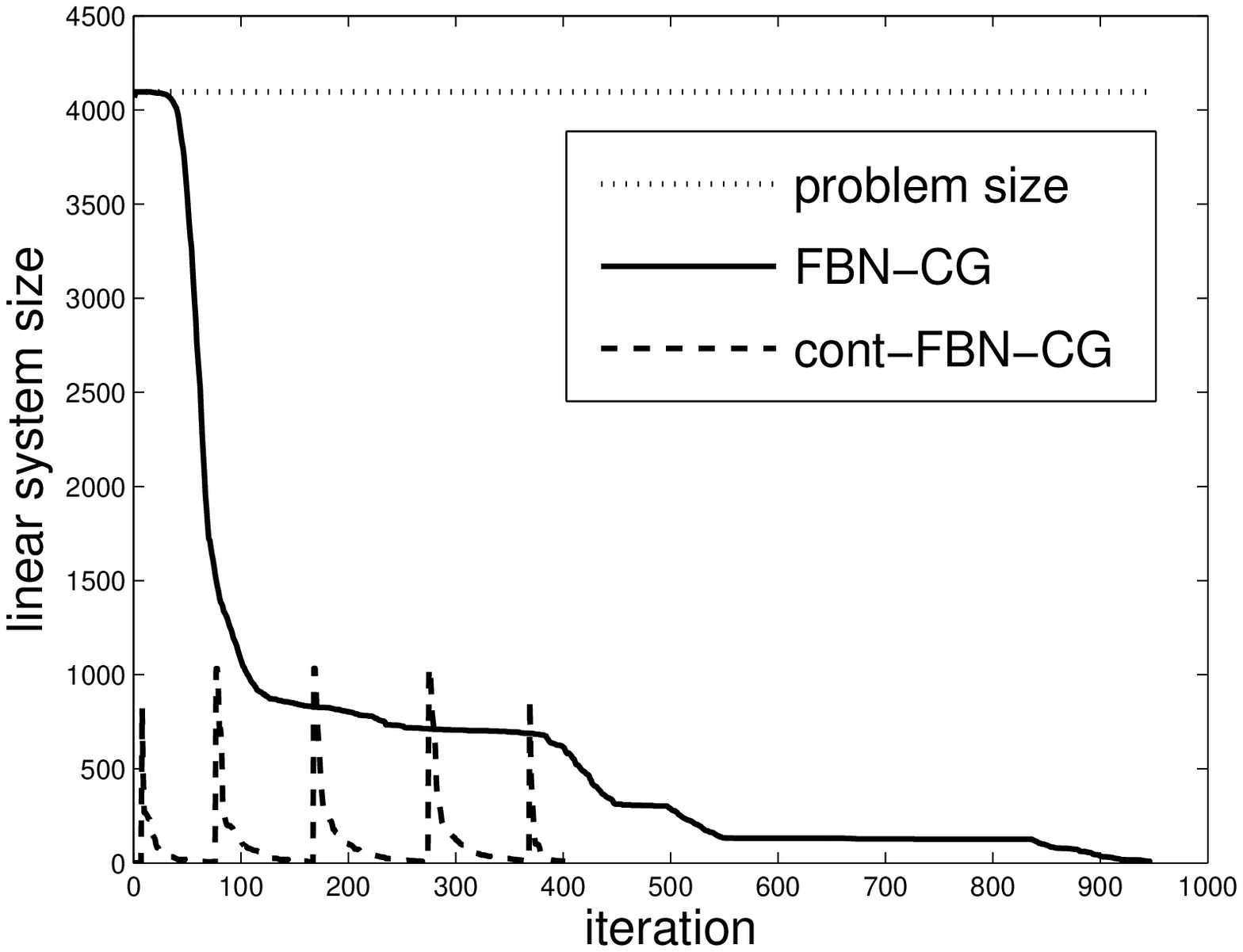}} &
\subfloat[\emph{spear\_inst\_151}]{\includegraphics[width=0.49\textwidth]{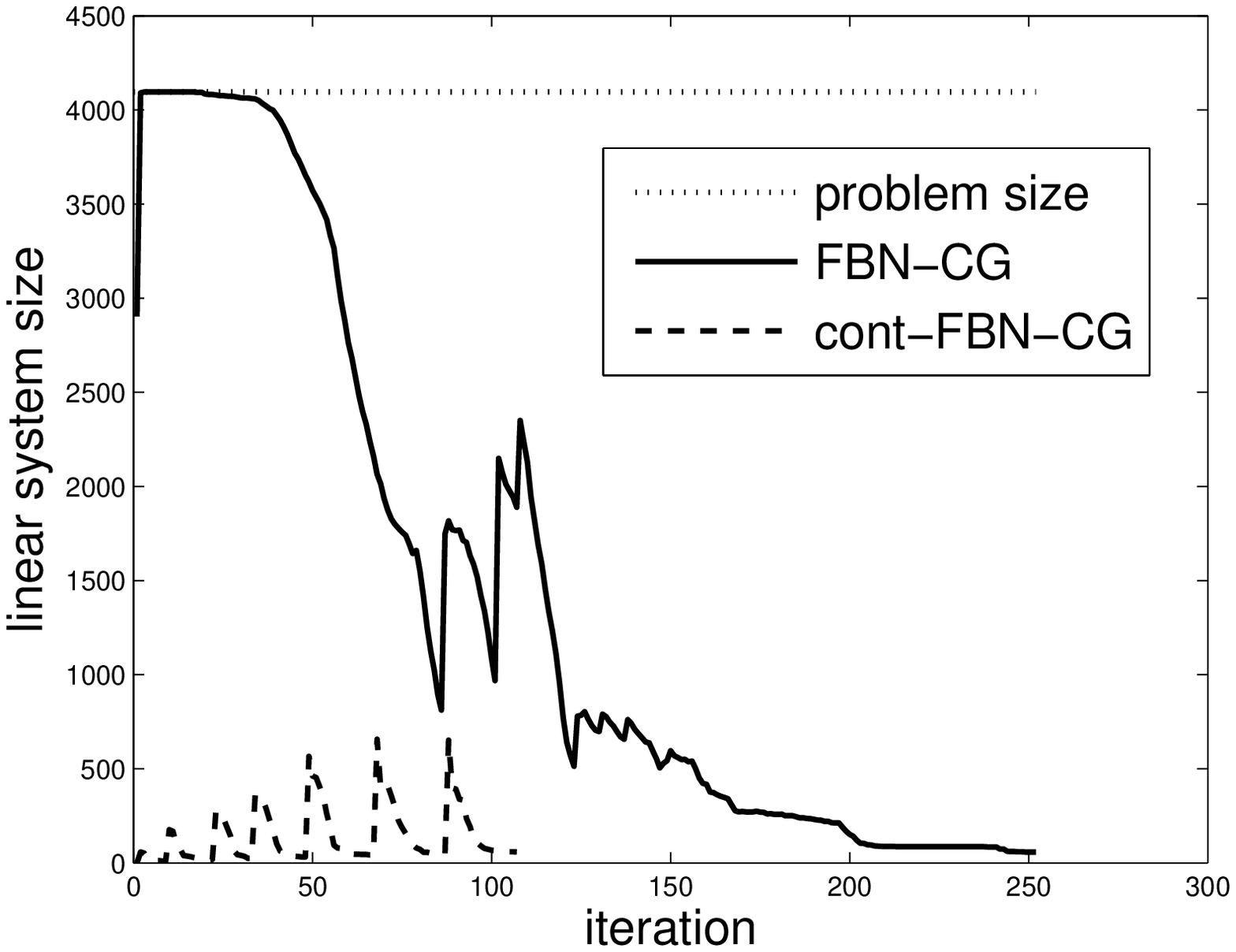}}
\end{tabular}
\caption{$\ell_1$-regularized least squares. Size of the linear system solved, by FBN-CG with and without warm-starting, compared to the full problem size.}
\label{fig:l1ls_3}
\end{figure}


\subsection{\texorpdfstring{$\ell_1$}{l1}-regularized logistic regression}
This is another example of sparse fitting problem, although here the solution is used to perform binary classification.
The composite objective function consists of
$$ f(x) = \sum_{i=1}^m\log(1+e^{-a_{i}'x}),\qquad g(x) = \lambda\|x_{[n-1]}\|_1,$$
and again the $\ell_1$-regularization enforces sparsity in the solution.
We have
$$(\prox_{\gamma g}(x))_i = \begin{cases}(\sign(x_i)(|x_i|-\lambda\gamma)_{+})_i, & i = 1,\ldots,n+1, \\ x_i & i = n.\end{cases}$$
Let $A\in\Re^{m\times n}$ be the feature matrix with rows $a_{i}$ having the trailing feature (the \emph{bias} term) equals to one.
If we set $\sigma(x) = (1+e^{-Ax})^{-1}$ and let $\Sigma(x) = \diag(\sigma(x)\circ (1-\sigma(x)))$,
then the Newton system~\eqref{eq:RegNewtSys} is
$$ \begin{bmatrix} I_{|\bar\alpha|} & \\ \gamma A_{\alpha}'\Sigma(x)A_{\bar\alpha} & \gamma A_{\alpha}'\Sigma(x)A_{\alpha} \end{bmatrix} d = P_{\gamma}(x) - x, $$
where this time
\begin{align*}
\alpha &= \left\{i\ \left.\right|\ \left|\left[x-\gamma\nabla f(x)\right]_i\right| > \gamma\lambda\right\}\cup\left\{n\right\},\\
\bar\alpha &= \left\{1,\ldots, n\right\}\setminus \alpha.
\end{align*}
We compared FBN-CG to the accelerated FBS~\cite{nesterov2007gradient}.
A continuation technique, similar to what described for the previous example, is employed in order to keep $|\alpha|$ small.
As in the previous example, an approximate solution to the problem was first
computed by means of extensive runs of one of the methods, and then the algorithms were set to stop
once at a relative distance of $\epsilon_r = 10^{-8}$ from it.
Table~\ref{tbl:l1lr_1} shows
how the methods scale with the number of features $n$,
for sparse random datasets with $m = n/10$ and $\approx 50$ nonzero features
per row. The datasets were generated according to what described
in~\cite[Sec.~4.2]{koh2007interior}. It is apparent how FBN-CG improves with
respect to the accelerated version of forward-backward splitting.

\begin{table}[tp!]
\center
\begin{tabular}{|r|rr|rr|rr|rr|} 
\hline 
    & \multicolumn{2}{c|}{FBN-CG I} & \multicolumn{2}{c|}{FBN-CG II} & \multicolumn{2}{c|}{Accel. FBS} \\
\hline \hline
$n$ & time & iter. & time & iter. & time & iter.\\
\hline 
100 & 0.04 & 51.1 & 0.05 & 57.3 & 0.06 & 292.4\\
215 & 0.05 & 52.8 & 0.06 & 61.0 & 0.11 & 462.1\\
464 & 0.06 & 54.4 & 0.09 & 69.4 & 0.18 & 647.2\\
1000 & 0.08 & 62.2 & 0.12 & 74.4 & 0.33 & 962.3\\
2154 & 0.27 & 98.8 & 0.35 & 108.2 & 0.82 & 1553.2\\
4641 & 0.95 & 151.1 & 0.94 & 142.2 & 3.58 & 2451.3\\
10000 & 2.40 & 217.7 & 2.54 & 207.0 & 9.36 & 3553.6\\
\hline 
\end{tabular} 
\caption{$\ell_1$-regularized logistic regression. Average running time (in seconds) and average number of iterations, for random datasets
with $m = n/10$ and increasing $n$, $\lambda = 1$.}
\label{tbl:l1lr_1}
\end{table}


\subsection{Matrix completion}
We consider the problem of recovering the entries of a matrix, which is known to have small rank, from a sample of them. One may refer to~\cite{candes2009exact} for
a detailed theoretical analysis of the problem.
Since we are now dealing with matrix variables, we conveniently adopt the notation of \emph{vector representation} of the matrix $X$, denoted by $\vec(X)$, 
\ie the $mn$-dimensional vector obtained by stacking the columns of $X$.
The problem is formulated in a composite form as
$$ f(X) = \frac{1}{2}\|\mathcal{A}(X)-b\|^2,\qquad g(X) = \lambda\|X\|_*.$$
The linear mapping $\mathcal{A}:\Re^{m\times n}\to\Re^k$ is represented as a $k$-by-$mn$
matrix $A$ acting on $\vec(X)$. The problem is nothing more than a least squares problem with a nuclear norm regularization term, having
$\nabla f(X) = A'(A\vec(X)-b)$ and $\nabla^2 f(X) = A'A$.
For a matrix completion task, matrix $A$ is a binary matrix that selects $k$ elements from $X$.
Hence $\nabla^2 f(X)$ is actually diagonal, with $k$ nonzero elements:
$$ A'A = \diag(h_1,\ldots,h_{mn}),\quad h_i = \begin{cases}1 & i \mbox{ is selected by } A,\\0 & \mbox{otherwise}.\end{cases}$$
The proximal mapping associated with $g = \|\cdot\|_*$ is the soft-thresholding applied to the singular values of the matrix argument. Its $B$-subdifferential elements
act on $m$-by-$n$ matrices as expressed in \eqref{eq:JacNS}: if we consider, again, vector representations
the linear mapping $P$ is explicitly expressed by some symmetric and positive semi-definite matrix $Q\in\Re^{mn\times mn}$ with eigenvalues in the interval $[0,1]$.
Hence we can express \eqref{eq:RegNewtSys} as follows:
\begin{equation} (G-GQG + \delta I)\vec(D) = -G\vec(X-P_{\gamma}(X)),\label{eq:NewtonSystemMC1}\end{equation}
where $$ G = I-\gamma \nabla^2 f(x) = I_{mn} - \gamma A'A = \diag(g_1,\ldots,g_{mn}), $$
has diagonal elements $1-\gamma$ and $1$. Note however that we don't need to form the system \eqref{eq:NewtonSystemMC1} order compute residuals and carry out CG iterations, as matrix $Q$ is indeed
very large and dense. Instead, one can observe that pre-multiplication of $\vec(D)$ by a diagonal matrix $G$ is equivalent to the Hadamard product $\widehat{G}\circ D$, where
$$ \widehat{G} = \begin{bmatrix}g_1 & g_{m+1} & \cdots & g_{(n-1)m+1} \\ \vdots & \vdots & & \vdots \\ g_m & g_{2m} & \cdots & g_{nm} \end{bmatrix}. $$
Furthermore, with arguments similar to the ones in~\cite{zhao2010newton}, the computational effort needed to evaluate $P$ can be drastically reduced
due to the sparsity pattern of matrices $\Omega_1, \Omega_2, \Omega_3$ in \eqref{eq:JacNS}.
Hence it is convenient to compute residuals according to the following rewriting of \eqref{eq:NewtonSystemMC1}:
\begin{equation} \widehat{G}\circ D-\widehat{G}\circ P(\widehat{G}\circ D) + \delta D = -\widehat{G}\circ(X-P_{\gamma}(X)).\label{eq:NewtonSystemMC2}\end{equation}
Even in this case, as in the previous examples, we can warm start our methods by approximately solving it for $\lambda_{\mbox{\tiny max}} \geq \lambda > \lambda_{0}$ and updating $\lambda$
in a continuation scheme until the final stage in which $\lambda = \lambda_{0}$.

We considered the accelerated proximal gradient
with line search (APGL)~\cite{toh2010accelerated} and the linearized alternating
direction method (LADM)~\cite{yang2013linearized} in performing our tests. Both the methods also
implement continuation on their parameters. Table~\ref{tbl:mc_1} shows the
average performance, in terms of number of iterations and SVD computations,
on random matrices generated according to~\cite{toh2010accelerated}.
FBN-CG always succeeds at finding a low-error solution within a moderate number
of iterations and SVD computations, which is not the case for LADM. Regarding
APGL, it is worth noticing that it takes advantage from different
acceleration techniques for this specific problem, which we have not considered
for our algorithm. The drawback of our method is that at every iteration,
the computation of \eqref{eq:JacNS} requires a full SVD as opposed to a
decomposition in reduced form. Whether this can be avoided, and how this would
affect the overall method, requires further investigation.

\begin{table}
\center
\begin{tabular}{|r||r|r|r|r|r|r|}
\hline
& $m$ ($=n$) & density & iterations  & SVDs & error     \\ 
\hline
\multirow{3}{*}{FBN-CG I} & 100 & 0.56 & 67.3 & 86.2 & 6.89e-04\\
& 200 & 0.35 & 76.8 & 100.3 & 3.56e-04\\
& 500 & 0.20 & 83.8 & 96.8 & 1.92e-04\\
\hline
\multirow{3}{*}{FBN-CG II} & 100 & 0.56 & 54.1 & 126.1 & 6.89e-04\\
& 200 & 0.35 & 65.6 & 153.3 & 3.56e-04\\
& 500 & 0.20 & 71.0 & 151.2 & 1.92e-04\\
\hline
\multirow{3}{*}{APGL} & 100 & 0.56 & 92.4 & 92.4 & 5.94e-04\\
& 200 & 0.35 & 94.9 & 94.9 & 3.56e-04\\
& 500 & 0.20 & 67.3 & 67.3 & 1.92e-04\\
\hline
\multirow{3}{*}{LADM} & 100 & 0.56 & 183.2 & 183.2 & 4.58e-03\\
& 200 & 0.35 & 494.2 & 494.2 & 7.57e-03\\
& 500 & 0.20 & 1000.0 & 1000.0 & 2.70e-02\\
\hline
\end{tabular}
\caption{Matrix completion. Average performance on $10$ randomly generated instances $M$ with $\mbox{rank}(M) = 10$, $\lambda = 10^{-2}$.
The density column refers to the fraction of observed coefficients. APGL and LADM require one SVD per iteration.
The error reported is $\|X-M\|_F/\|M\|_F$, the relative distance $X$, the computed solution, and the original matrix $M$.}
\label{tbl:mc_1}
\end{table}

\section{Conclusions and Future Work}
In this paper we presented a framework, based on the continuously differentiable function~\eqref{eq:Penalty} which we called
\emph{forward-backward envelope (FBE)},
to address a wide class of nonsmooth convex optimization problems in composite form.
Problems of this form arise in many fields such as control, signal and image processing, system identification and machine learning.
Using tools from nonsmooth analysis we derived two algorithms, namely FBN-CG I and II, that are Newton-like methods minimizing the FBE,
for which we proved fast asymptotic convergence rates.
Furthermore, Theorems~\ref{th:ComplPNM},~\ref{th:PGNMbnds1} and~\ref{th:PGNMbnds2} provide global complexity estimates,
making the algorithms also appealing for real-time applications. The considered approach makes it possible to exploit the sparsity patterns of many problems in the vicinity of the solution,
so that the resulting Newton system is usually of small dimension for many significant problems. This also implies that the algorithms can favorably take
advantage of warm-starting techniques. Our computational experience supports the theoretical results, and shows how in some scenarios our method
challenges other well known approaches.

The framework we introduced opens up the possibility of extending many existing and well known algorithms, originally introduced for smooth unconstrained optimization,
to the nonsmooth or constrained case. This is the case for example of Newton methods based on a trust-region approach, as well as quasi-Newton methods.
Future work includes embedding the Newton iterations in accelerated versions of the forward-backward splitting, in order to obtain better global convergence rates.
Finally, the extension of the framework to the nonconvex case (\ie to the case in which the smooth term $f$ in~\eqref{eq:GenProb} is nonconvex) can also be considered
in order to address a wider range of applications.

\bibliographystyle{IEEEtran}

\bibliography{IEEEabrv,NEWTON_CDC_bib}

\appendix

\iftoggle{svver}{\section*{Appendix}\label{ap:AppA}}{\section{}\label{ap:AppA}}
We provide results that are used throughout the paper and all the
omitted proofs. The following result is useful in bounding the eigenvalues of the linear
Newton approximation of $F_\gamma$, and is required by Theorem~\ref{th:ProxPropQuad} and Proposition~\ref{prop:PSDHess}.
\begin{lemma}\label{lem:eigen}
If $Q\in\Ss_+^n$ and $\mu_f=\lambda_{\min}(Q)$, $L_f=\lambda_{\max}(Q)$ then 
\[
\lambda_{\min}(Q(I-\gamma Q))=\begin{cases}
\mu_f(1-\gamma\mu_f),&\mathrm{ if }\ 0<\gamma\leq 1/(L_f+\mu_f),\\
L_f(1-\gamma L_f),&\mathrm{ if }\ 1/(L_f+\mu_f)\leq\gamma< 1/L_f.
\end{cases}
\]
\end{lemma}
\begin{proof}
Since $Q$ is symmetric positive semidefinite, there exists an invertible matrix $S\in\Rr^{n\times n}$ such that $Q=SJS^{-1}$, where $J=\diag(\lambda_1(Q),\ldots,\lambda_n(Q))$.
Therefore,
\begin{align*}
Q(I-\gamma Q)&=SJS^{-1}(I-\gamma SJS^{-1})\\
&=SJS^{-1} S(I-\gamma J)S^{-1}\\
&=SJ(I-\gamma J)S^{-1},
\end{align*}
and the eigenvalues of $Q(I-\gamma Q)$ are exactly 
$$\lambda_1(Q)(1-\gamma\lambda_1(Q)),\ldots,\lambda_n(Q)(1-\gamma\lambda_n(Q)).$$
Next, consider the minimization problem $\min_{\lambda\in [\mu_f,L_f]}\phi(\lambda)\eqdef\lambda(1-\gamma\lambda)$.
Since $\gamma$ is positive, $\phi$ is concave and the minimum is attained either at $\mu_f$ or $L_f$. The proof finishes by noticing that
$$\mu_f(1-\gamma\mu_f)\leq L_f(1-\gamma L_f) \Leftrightarrow \gamma\in\left(0,1/(L_f+\mu_f)\right).$$
\iftoggle{svver}{\qed}{}
\end{proof}

The next result gives condition for the Lipschitz-continuity of $P_\gamma$ and $G_\gamma$,
and is needed by Theorem~\ref{th:ProxPropQuad} to obtain the Lipschitz constant of
$\nabla F_\gamma$ in the case where $f$ is quadratic, and by Theorem~\ref{eq:PNMconvRate}
and~\ref{eq:PGNMconvRate} in order to assess the local convergence properties of
Algorithm~\ref{al:PNM} and~\ref{al:PGNM}.
\begin{lemma}\label{le:zNonExp}
Suppose that $\gamma<1/L_f$. Then $P_\gamma:\Re^n\to\Re^n$ is nonexpansive, i.e.,
\begin{equation}\label{eq:Pnonexp}
\|P_{\gamma}(x)-P_{\gamma}(y)\|\leq\|x-y\|,
\end{equation}
and $G_\gamma:\Re^n\to\Re^n$ is $(2/\gamma)$-Lipschitz continuous, i.e.,
\begin{equation}
\|G_{\gamma}(x)-G_{\gamma}(y)\|\leq2/\gamma\|x-y\|.
\end{equation}
\end{lemma}
\begin{proof}
On one hand we know that $\prox_{\gamma g}$ is firmly nonexpansive~\cite{moreau1965proximiteet},
therefore $\prox_{\gamma g}$ is a $1/2$-averaged operator~\cite[Rem. 4.24(iii)]{bauschke2011convex}.
On the other hand, being $\nabla f$ the Lipschitz continuous gradient of a convex function, it is $1/L_f$-cocoercive.
Therefore, since $\gamma<1/L_f$, the operator $x\to x-\gamma\nabla f(x)$ is $\gamma L_f/2$-averaged~\cite[Prop. 4.33]{bauschke2011convex}. Since $P_\gamma$ is the composition of two averaged operators, it is an averaged operator as well~\cite[Prop. 4.32]{bauschke2011convex}. By~\cite[Rem. 4.24(i)]{bauschke2011convex} this implies that $P_\gamma$ is nonexpansive, proving~\eqref{eq:Pnonexp}. Next, consider any $x,y\in\Re^n$:
\begin{align*}
\|G_\gamma(x)-G_\gamma(y)\|&\leq 1/\gamma\|P_\gamma(x)-P_{\gamma}(y)-(x-y)\|\\
&\leq 1/\gamma\left(\|P_\gamma(x)-P_{\gamma}(y)\|+\|x-y\|\right)\\
&\leq 2/\gamma\|x-y\|.
\end{align*}
\iftoggle{svver}{\qed}{}
\end{proof}

The following proposition is an extension of \cite[Lemma 2.3]{beck2009fast}
that handles the case where $f$ can be strongly convex.
\begin{proposition}\label{prop:ProxBasic}
For any  $\gamma\in (0,1/L_f]$, $x\in\Re^n$, $\bar{x}\in\Re^n$
\begin{equation}\label{eq:ProxBasic}
F(x)\geq F(P_\gamma(\bar{x}))+G_{\gamma}(\bar{x})'(x-\bar{x})+\tfrac{\gamma}{2}\|G_{\gamma}(\bar{x})\|^2+\tfrac{\mu_f}{2}\|x-\bar{x}\|^2.\nonumber
\end{equation}
\end{proposition}

\begin{proof}
For any $x\in\Re^n$, $\bar{x}\in\Re^n$ we have
\begin{align*}
F(x)    & \geq f(\bar{x})+\nabla f(\bar{x})'(x-\bar{x})+\tfrac{\mu_f}{2}\|x-\bar{x}\|^2\\
        & \phantom{\geq f(\bar{x})}+g(P_\gamma(\bar{x}))+(G_{\gamma}(\bar{x})-\nabla f(\bar{x}))'(x-P_\gamma(\bar{x}))\\
        & = f(\bar{x})+g(P_\gamma(\bar{x}))-\nabla f(\bar{x})'(\bar{x}-P_\gamma(\bar{x}))+G_{\gamma}(\bar{x})'(x-P_\gamma(\bar{x}))+\tfrac{\mu_f}{2}\|x-\bar{x}\|^2\\
        & = F_\gamma(\bar{x})-\tfrac{\gamma}{2}\|G_\gamma(\bar{x})\|^2+G_{\gamma}(\bar{x})'(\bar{x}-P_{\gamma}(\bar{x}))+G_{\gamma}(\bar{x})'(x-\bar{x})+\tfrac{\mu_f}{2}\|x-\bar{x}\|^2\\
        & = F_\gamma(\bar{x})-\tfrac{\gamma}{2}\|G_\gamma(\bar{x})\|^2+\gamma\|G_{\gamma}(\bar{x})\|^2+G_{\gamma}(\bar{x})'(x-\bar{x})+\tfrac{\mu_f}{2}\|x-\bar{x}\|^2\\
        & \geq F(P_\gamma(\bar{x}))+\tfrac{\gamma}{2}(2-\gamma L_f)\|G_\gamma(\bar{x})\|^2+G_{\gamma}(\bar{x})'(x-\bar{x})+\tfrac{\mu_f}{2}\|x-\bar{x}\|^2.
\end{align*}
The first inequality follows by strong convexity of $f$ and $G_{\gamma}(\bar{x})-\nabla f(\bar{x})\in\partial g(P_{\gamma}(\bar{x}))$, 
the equality by the definition of $F_\gamma$ and the final inequality by Theorem~\ref{Th:PropFg}(\ref{prop:LowBnd}).
The result follows by noticing that $\gamma\in (0,1/L_f]$ implies $2-\gamma L_f\geq 1$.
\iftoggle{svver}{\qed}{}
\end{proof}
An immediate result of Proposition~\ref{prop:ProxBasic} is the following.
\begin{corollary}\label{prop:GradLowBnd}
For any  $\gamma\in (0,1/L_f]$,  $x\in\Re^n$, it holds
\[
\|G_{\gamma}(x)\|^2\geq 2\mu_f(F(P_{\gamma}(x))-F_\star).
\]
\end{corollary}

\begin{proof}
According to Proposition~\ref{prop:ProxBasic}, if  $\gamma\in (0,1/L_f]$ then for any $x,\bar{x}\in\Re^n$ we certainly have
\begin{equation}\label{eq:StronConvzbound}
F(x)\geq F(P_\gamma(\bar{x}))+G_{\gamma}(\bar{x})'(x-\bar{x})+\tfrac{\mu_f}{2}\|x-\bar{x}\|^2.
\end{equation}
Minimizing both sides with respect to $x$ we obtain $F_\star$ for the left hand side
and $x=\bar{x}-\mu_f^{-1}G_{\gamma}(\bar{x})$ for the right hand side.
Substituting in \eqref{eq:StronConvzbound} we obtain
\begin{align*}
F_\star&\geq F(P_{\gamma}(\bar{x}))-\tfrac{1}{2\mu_f}\|G_{\gamma}(\bar{x})\|^2.
\end{align*}
\iftoggle{svver}{\qed}{}
\end{proof}

The next proposition is useful for proving the global linear convergence rate of
Algorithm~\ref{al:PGNM}, in the case of $f$ strongly convex,
cf. Theorem~\ref{th:PGNMbnds2}.
\begin{proposition}\label{prop:LipLowBndDist}
For any $x\in\Re^n$, $x_\star\in X_\star$ and $\gamma\in (0,1/L_f]$
\[
F(P_{\gamma}(x))-F_\star\leq\tfrac{1}{2\gamma}(1-\gamma\mu_f)\|x-x_\star\|^2.
\]
\end{proposition}

\begin{proof}
By definition of $F_\gamma$ we have
\begin{align*}
F_\gamma(x)&{=}\min_{z\in\Re^n}\left\{f(x){+}\nabla f(x)'(z-x)+g(z){+}\tfrac{1}{2\gamma}\|z-x\|^2\right\}\\
&\leq f(x){+}\nabla f(x)'(x_\star-x)+g(x_\star)+\tfrac{1}{2\gamma}\|x_\star-x\|^2\\
&\leq f(x_\star)+g(x_\star)-\tfrac{\mu_f}{2}\|x-x_\star\|^2+\tfrac{1}{2\gamma}\|x_\star-x\|^2,
\end{align*}
where the second inequality follows from (strong) convexity of $f$. The proof finishes by invoking Theorem~\ref{Th:PropFg}(\ref{prop:LowBnd}).
\iftoggle{svver}{\qed}{}
\end{proof}

Hereafter we provide the proofs omitted in Sections~\ref{sec:LNA} and~\ref{sec:FBNCG}.
\iftoggle{svver}{\paragraph{Proof of Proposition~\ref{prop:LNAprops1}\\\\}\label{par:proofLNA1}}
{\begin{proof}[Proof of Proposition~\ref{prop:LNAprops1}]\label{par:proofLNA1}}
Let $T(x)=x-\gamma\nabla f(x)$. Then $P_\gamma$ can be expressed as the composition
of mappings $\prox_{\gamma g}$ and $T$, \ie $P_\gamma(x)=\prox_{\gamma g}(T(x))$.
Since $\prox_{\gamma g}$ is (strongly) semismooth at $T(x_\star)$ we have that
$\partial_C\prox_{\gamma g}$ is a (strong) LNA for $\prox_{\gamma g}$ at $T(x_\star)$.
On the other hand, since $T$ is twice continuously differentiable, its Jacobian
$\nabla T(x)=I-\gamma\nabla^2 f(x)$ is a  LNA of $T$ at $x_\star$.
If in addition $\nabla^2 f$ is Lipschitz continuous around $x_\star$ then
$\nabla T$ is a strong LNA of $T$ at $x_\star$~\cite[Prop. 7.2.9]{facchinei2003finite}.
Invoking~\cite[Th. 7.5.17]{facchinei2003finite} we have that
$$\mathscr{P}_\gamma(x)=\{P(I-\gamma\nabla^2 f(x))\ |\ P\in\partial_C(\prox_{\gamma g})(x-\gamma\nabla f(x))\},$$ 
is a (strong) LNA of $P_\gamma$ at $x_\star$. 

Next consider $G_\gamma(x)=\gamma^{-1}(x-P_\gamma(x))$. Applying~\cite[Cor. 7.5.18(a)(b)]{facchinei2003finite}
we have 
$$\mathscr{G}_\gamma(x)=\{\gamma^{-1}(I-V)\ |\ V\in\mathscr{P}_\gamma(x)\},$$
is a (strong) LNA for $G_\gamma$ at $x_\star$.
Reinterpreting $\hat{\partial}^2F_\gamma(x)$ with the current notation,
\begin{align*}
\hat{\partial}^2F_\gamma(x)&=\{(I-\gamma\nabla^2f(x))Z\ |\ Z\in\mathscr{G}_\gamma(x)\}.
\end{align*}
Therefore, for any $H\in\hat{\partial}^2F_\gamma(x)$
\begin{align*}
\|\nabla F_\gamma(x)+H(x_\star-x)-\nabla F_\gamma(x_\star)\|&=\|(I-\gamma\nabla^2f(x))(G_\gamma(x)+Z(x-x_\star)-G_\gamma(x_\star))\|\\
&\leq\|G_\gamma(x)+Z(x-x_\star)-G_\gamma(x_\star)\|,
\end{align*}
where the equality follows by $0=\nabla F_\gamma(x_\star)=(I-\gamma\nabla f^2(x_\star))G_\gamma(x_\star)$,
and the inequality by $\gamma\in (0,1/L_f)$. Since $\mathscr{G}_\gamma$ is a (strong) LNA of $G_\gamma$, the last term is $o(\|x-x_\star\|)$
(and $O(\|x-x_\star\|^2)$ in the case where $\nabla^2 f$ is locally Lipschitz continuous).
This shows that $\hat{\partial}F_\gamma$ is a (strong) LNA of $\nabla F_\gamma$ at $x_\star$.
\iftoggle{svver}{\qed}{\end{proof}}

\iftoggle{svver}{\paragraph{Proof of Proposition~\ref{prop:PSDHess}\\\\}\label{par:proofLNA2}}
{\begin{proof}[Proof of Proposition~\ref{prop:PSDHess}]}
Any $H\in\hat{\partial}^2 F_\gamma(x)$ can be expressed as
$$H=\gamma^{-1}(I-\gamma\nabla^2 f(x))-\gamma^{-1}(I-\gamma\nabla^2 f(x))P(I-\gamma\nabla^2 f(x))$$
for some $P\in\partial_C(\prox_{\gamma g})(x-\gamma\nabla f(x))$. Obviously, recalling Theorem~\ref{th:JacProx}, $H$ is a symmetric matrix. We have
\begin{align*}
d'Hd & =\gamma^{-1}d'(I-\gamma\nabla^2 f(x))d-\gamma^{-1}d'(I-\gamma\nabla^2 f(x))P(I-\gamma\nabla^2 f(x))d\\
&\geq \gamma^{-1}d'(I-\gamma\nabla^2 f(x))d-\gamma^{-1}\|(I-\gamma\nabla^2 f(x))d\|^2\\
&= d'(I-\gamma\nabla^2 f(x))\nabla^2 f(x)d\\
&\geq  \min\{(1-\gamma\mu_f)\mu_f,(1-\gamma L_f)L_f\}\|d\|^2,
\end{align*}
where the first inequality follows by Theorem~\ref{th:JacProx} and the second by Lemma~\ref{lem:eigen}.
On the other hand
\begin{align*}
d'Hd & =\gamma^{-1}d'(I-\gamma\nabla^2 f(x))d-\gamma^{-1}d'(I-\gamma\nabla^2 f(x))P(I-\gamma\nabla^2 f(x))d\\
&\leq\gamma^{-1}d'(I-\gamma\nabla^2 f(x))d\\
&\leq \gamma^{-1}(1-\gamma\mu_f)\|d\|^2,
\end{align*}
where the first inequality follows by Theorem~\ref{th:JacProx}.
\iftoggle{svver}{\qed}{\end{proof}}

\iftoggle{svver}{\paragraph{Proof of Lemma~\ref{lem:sharpMin}\\\\}\label{par:proofLNA3}}
{\begin{proof}[Proof of Lemma~\ref{lem:sharpMin}]}
It suffices to prove that
$\|x-x_\star\|\leq c\|\nabla F_\gamma(x)\|$, for all
$x\ \mathrm{with}\ \|x-x_\star\|\leq\delta$ and some positive
$c$, $\delta$. The result will then follow, since
$\|\nabla F_\gamma(x)\|=\|(I-\gamma\nabla^2f(x))G_\gamma(x)\|\leq \|G_\gamma(x)\|$,
for $\gamma\in (0,1/L_f)$. For the sake of contradiction assume that there
exists a sequence of vectors $\{x^k\}$ converging to $x_\star$ such that
$x^k\neq x_\star$ for every $k$ and 
\begin{equation}\label{eq:contra}
\lim_{k\to\infty}\frac{\nabla F_\gamma(x^k)}{\|x^k-x_\star\|}=0.
\end{equation}
The assumptions of the lemma guarantee through Proposition~\ref{prop:LNAprops1} that $\hat{\partial}^2F_\gamma$ is a LNA of $\nabla F_\gamma$ at $x_\star$, therefore
$$0=\lim_{k\to\infty}\frac{\nabla F(x^k)+H^k(x_\star-x^k)-\nabla F_\gamma(x^\star)}{\|x^k-x_\star\|}=\lim_{k\to\infty}\frac{H^k(x_\star-x^k)}{\|x^k-x_\star\|},$$
where the second equality follows from~\eqref{eq:contra}. This implies that
$$\lim_{k\to\infty}\frac{(x_\star-x^k)'H^k(x_\star-x^k)}{\|x^k-x_\star\|^2}=0.$$
But since $\hat{\partial}^2F_\gamma$ is compact-valued and outer semicontinuous at $x_\star$, and $\{x^ k\}$ converges to $x_\star$, the nonsingularity assumption on the elements of $\hat{\partial}^2 F_\gamma(x_\star)$ implies through~\cite[Lem.~7.5.2]{facchinei2003finite} that for sufficiently large $k$, the smallest eigenvalue of $H^ k$ is minorized by a positive number. Therefore the above limit must be positive, reaching to a contradiction. Uniqueness follows from the fact that the set of zeros of $\nabla F_\gamma$ is equal to the set of optimal solutions of~\eqref{eq:GenProb}, through Theorem~\ref{Th:PropFg}\eqref{prop:DerPen}.
\iftoggle{svver}{\qed}{\end{proof}}
\iftoggle{svver}{\paragraph{Proof of Theorem~\ref{th:ComplPNM}\\\\}\label{proof:ComplPNM}}
{\begin{proof}[Proof of Theorem~\ref{th:ComplPNM}.]}
Since $\mu_f>0$ and $\zeta=0$, using Proposition~\ref{prop:PSDHess}, Eq.~\eqref{eq:CGprop} gives
\begin{equation}\label{eq:Dir1}
\nabla F_\gamma(x^ k)'d^ k\leq -c_1\|\nabla F_\gamma(x^ k)\|^2.
\end{equation}
where $c_1=\frac{\gamma}{(1-\gamma\mu_f)}$
while Eq.~\eqref{eq:boundd} gives
\begin{equation}\label{eq:Dir}
\|d^ k\|\leq c_2\|\nabla F_\gamma(x^ k)\|
\end{equation}
where $c_2=(\eta+1)/\xi_1$, $\xi_1\eqdef\min\left\{(1-\gamma\mu_f)\mu_f,(1-\gamma L_f)L_f\right\}$.
Using Eqs.~\eqref{eq:Armijo},~\eqref{eq:Dir1}, step $\tau_k=2^{-i_k}$ satisfies
$$F_\gamma(x^k+\tau_kd^k)-F_\gamma(x^k)\leq-\sigma\tau_kc_1\|\nabla F_\gamma(x^k)\|^2.$$
Due to Theorem~\ref{th:ProxPropQuad}, $\nabla F_\gamma$ is Lipschitz continuous, therefore using the descent Lemma~\cite[Prop. A.24]{bertsekas1999nonlinear}
\begin{align}
F_\gamma(x^k+2^{-i}d^k)-F_\gamma(x^k)&\leq 2^{-i}\nabla F_\gamma(x^k)'d^k+\tfrac{L_{F_\gamma}}{2}2^{-2i}\|d^k\|^2\nonumber\\
&\leq-2^{-i}c_1\|\nabla F_{\gamma}(x^k)\|^2+\tfrac{L_{F_{\gamma}}}{2}{c^2_2}2^{-2i}\|\nabla F_{\gamma}(x^k)\|^2 \nonumber\\
&\leq-2^{-i}c_1(1-\tfrac{L_{F_{\gamma}}}{2}\tfrac{c^2_2}{c_1}2^{-i})\|\nabla F_{\gamma}(x^k)\|^2\label{eq:LipDecrease}
\end{align}
where the second inequality follows by~\eqref{eq:Dir}.
 Let $i_{\min}$ be the first index $i$ for which $1-\tfrac{L_{F_\gamma}}{2}\tfrac{c^2_2}{c_1}2^{-i}\geq\sigma$, \ie
\begin{subequations}\label{eq:Arm}
\begin{align}
1-\tfrac{L_{F_\gamma}}{2}\tfrac{c^2_2}{c_1}2^{-i}&<\sigma,\quad 0\leq i< i_{\min}\label{eq:Arm1}\\
1-\tfrac{L_{F_\gamma}}{2}\tfrac{c^2_2}{c_1}2^{-i_{\min}} &\geq\sigma\label{eq:Arm2}
\end{align}
\end{subequations}
From~\eqref{eq:Armijo},~\eqref{eq:LipDecrease} and~\eqref{eq:Arm} we conclude that $i_k\leq i_{\min}$, therefore $\tau_k\geq\hat{\tau}_{\min}$, where $\hat{\tau}_{\min}=2^{-i_{\min}}$, thus we have
\begin{equation}\label{eq:MinDecrease}
F_{\gamma}(x^k+\tau_kd^k)-F_{\gamma}(x^k)\leq-\sigma\hat{\tau}_{\min}c_1\|\nabla F_{\gamma}(x^k)\|^2
\end{equation}
 From Eq.~\eqref{eq:Arm1} we obtain 
$$\sigma > 1-\tfrac{L_{F_\gamma}}{2}\tfrac{c^2_2}{c_1}2^{-(i_{\min}-1)}=1-\tfrac{c^2_2}{c_1}L_{F_\gamma}2^{-i_{\min}}=1-\tfrac{c^2_2}{c_1}L_{F_\gamma}\hat{\tau}_{\min}$$
Hence 
\begin{equation}\label{eq:minStep}
\hat{\tau}_{\min}>\frac{1-\sigma}{L_{F_\gamma}}\frac{c_1}{c^2_2}.
\end{equation}
Subtracting $F_\star$ from both sides of~\eqref{eq:MinDecrease} and using~\eqref{eq:minStep}
\begin{equation}\label{eq:OnS}
F_\gamma(x^k+\tau_kd^k)-F_\star\leq F_\gamma(x^k)-F_\star-\tfrac{\sigma(1-\sigma)}{L_{F_\gamma}}\tfrac{c^2_1}{c^2_2}\|\nabla F_\gamma(x^k)\|^2.
\end{equation}
Since $F_\gamma$ is strongly convex (cf. Theorem~\ref{th:ProxPropQuad}) we have~\cite[Th. 2.1.10]{nesterov2003introductory} 
\begin{equation}\label{eq:strConvLow}
F_\gamma(x^{k})-F_\star\leq\frac{1}{2\mu_{F_\gamma}}\|\nabla F_\gamma(x^k)\|^2.
\end{equation}
Combining~\eqref{eq:OnS} and~\eqref{eq:strConvLow} we obtain
$$F_\gamma(x^{k+1})-F_\star\leq r_{F_\gamma}(F_\gamma(x^k)-F_\star)$$
where $r_{F_\gamma}= 1-\tfrac{2\mu_{F_\gamma}\sigma(1-\sigma)}{L_{F_\gamma}}\tfrac{c^2_1}{c^2_2}$, therefore
$$F_\gamma(x^{k})-F_\star\leq  r_{F_\gamma}^k(F_\gamma(x^0)-F_\star).$$
Using $F(P_\gamma(x^k))\leq F_\gamma(x^k)$ (cf. Theorem~\ref{Th:PropFg}\eqref{prop:LowBnd}) we arrive at~\eqref{eq:QuadRateF}.
Using~\cite[Th. 2.1.8]{nesterov2003introductory}  
$$(\mu_{F_\gamma}/2)\|x-x_\star\|^2\leq F_\gamma(x)-F_\star\leq (L_{F_\gamma}/2)\|x-x_\star\|^2$$
we obtain~\eqref{eq:QuadRatex}.
\iftoggle{svver}{\qed}{\end{proof}}

\iftoggle{svver}{\paragraph{Proof of Theorem~\ref{th:PGNMbnds1}\\\\}\label{proof:PGNMbnds1}}
{\begin{proof}[Proof of Theorem~\ref{th:PGNMbnds1}]}
If  $k\notin\mathcal{K}$ and $s_k=0$, then $F(x^{ k+1})=F(P_\gamma(x^{ k}))\leq F_\gamma (x^{ k})$, where the inequality follows from~\eqref{eq:LowBnd4Gamma}.
If $k\in\mathcal{K}$ or $s_k=1$, then $F(x^{ k+1})=F(P_{\gamma}(\hat{x}^k))\leq F_\gamma(\hat{x}^k)\leq F_\gamma(x^ k)$, where the first inequality uses~\eqref{eq:LowBnd4Gamma} while the second   uses the fact that $d^ k$ is a direction of descent for $F_\gamma$. Therefore, we  have 
\begin{equation}\label{eq:Comp1}
F(x^{k+1})\leq F_\gamma (x^{k}),\quad k\in\Nn.
\end{equation}
Next, for any $x\in\Re^n$
\begin{equation}\label{eq:Comp2}
F_\gamma(x)\leq\min_{z\in\Re^n}\left\{f(z)+g(z)+\tfrac{1}{2\gamma}\|z-x\|^2\right\}=F^{\gamma}(x),
\end{equation} 
where the inequality uses the convexity of $f$ (recall that $F^\gamma$ is the Moreau envelope of $F=f+g$). Combining~\eqref{eq:Comp1} with~\eqref{eq:Comp2}, we obtain $F(x^{k+1})\leq F^\gamma(x^k)$.
The rest of the proof is similar to \cite[Th. 4]{nesterov2007gradient}. In particular we have
\begin{align*}
F(x^{k+1})&\leq F^\gamma(x^k)=\min_{x\in\Re^n}\left\{F(x)+\tfrac{1}{2\gamma}\|x-x^k\|^2\right\}\\
&\leq\min_{0\leq\alpha\leq 1}\left\{F(\alpha x_\star+(1-\alpha)x^k)+\tfrac{\alpha^2}{2\gamma}\|x^k-x_\star\|^2\right\}\\
&\leq\min_{0\leq\alpha\leq 1}\left\{F(x^k)-\alpha(F(x^k)-F_\star)+\tfrac{R^2}{2\gamma}\alpha^2\right\},
\end{align*}
where the last inequality follows by convexity of $F$.
If  $F(x^0)-F_\star\geq R^2/\gamma$, then the optimal solution of the latter problem for $k=0$ is $\alpha=1$ and we obtain~
\eqref{eq:FirstStep}. Otherwise, the optimal solution is $\alpha=\frac{\gamma(F(x^k)-F_\star)}{R^2}\leq \frac{\gamma(F(x^0)-F_\star)}{R^2}\leq 1$ and we obtain
$$F(x^{k+1})\leq F(x^k)-\frac{\gamma(F(x^k)-F_\star)^2}{2R^2}.$$
Letting $\lambda_k=\frac{1}{F(x^k)-F_\star}$ the latter inequality is expressed as
$$\frac{1}{\lambda_{k+1}}\leq\frac{1}{\lambda_k}-\frac{\gamma}{2R^2\lambda_k^2}.$$
Multiplying both sides by $\lambda_{k}\lambda_{k+1}$ and rearranging 
\begin{align*}
\lambda_{k+1}\geq\lambda_k+\frac{\gamma}{2R^2}\frac{\lambda_{k+1}}{\lambda_k}\geq \lambda_k+\frac{\gamma}{2R^2}
\end{align*}
where the latter inequality follows from the fact that $\{F(x^{k})\}_{k\in\Nn}$ is nonincreasing (cf.~\eqref{eq:DesPGNM}). Summing up for $0,\ldots,k-1$ we obtain
$$\lambda_k\geq\lambda_0+\frac{\gamma}{2R^2}k\geq\frac{\gamma}{2R^2}(k+2)$$
where the last inequality follows by $F(x^0)-F_\star\leq R^2/\gamma$. Rearranging, we arrive at~\eqref{eq:kStep}.
\iftoggle{svver}{\qed}{\end{proof}}

\iftoggle{svver}{\paragraph{Proof of Theorem~\ref{th:PGNMbnds2}\\\\}\label{proof:PGNMbnds2}}
{\begin{proof}[Proof of Theorem~\ref{th:PGNMbnds2}]}
If  $k\notin\mathcal{K}$ and $s_k=0$, then $x^{ k+1}=P_\gamma(x^ k)$ and the decrease condition \eqref{eq:DesPGNM} holds.
Subtracting $F_\star$ from both sides and using Corollary~\ref{prop:GradLowBnd} we obtain
\begin{equation}\label{eq:StronConDec}
F(x^ k)-F_{\star}\geq (1+\gamma\mu_f)(F(x^{ k+1})-F_\star).
\end{equation}
If $k\in\mathcal{K}$ or $s_k=1$, we have $F(x^{ k+1})=F(P_\gamma(\hat{x}^ k))\leq F_\gamma(\hat{x}^ k)-\tfrac{\gamma}{2}\|G_\gamma(\hat{x}^ k)\|^2\leq F_\gamma({x}^ k)-\tfrac{\gamma}{2}\|G_\gamma(\hat{x}^ k)\|^2\leq F(x^ k)-\tfrac{\gamma}{2}\|G_\gamma(\hat{x}^k)\|^2$, where the first inequality follows from Theorem~\ref{Th:PropFg}(\ref{prop:LowBnd}), the second from~\eqref{eq:Armijo} and the descent property of $d^ k$ and the third one from Theorem~\ref{Th:PropFg}(\ref{prop:UppBnd}).  
Subtacting $F_\star$ from both sides
$$F(x^{ k+1})-F_\star+\tfrac{\gamma}{2}\|G_\gamma(\hat{x}^k)\|^2\leq F(x^k)-F^\star.$$
Using Corollary~\ref{prop:GradLowBnd}, we obtain $\|G_\gamma(\hat{x}^ k)\|^2\geq2\mu_f(F(P_\gamma(\hat{x}^ k)-F_\star)=2\mu_f(F(x^{ k+1})-F_\star)$. Combining the last two inequalities we  again obtain \eqref{eq:StronConDec}, which proves \eqref{eq:strC1}. Now, from \eqref{eq:StronConDec} we obtain
\begin{align}
F(x^{ k+1})-F_{\star}&\leq(1+\gamma\mu_f)^{- k}(F(x^1)-F_\star)\nonumber\\
&= (1+\gamma\mu_f)^{- k}(F(P_{\gamma}(x^0))-F_\star)\nonumber\\
&\leq\frac{1-\gamma\mu_f}{2\gamma(1+\gamma\mu_f)^{ k}}\|x^0-x_\star\|^2\label{eq:distX1},
\end{align}
where the equality comes from the fact that $s_0=0$ and the second inequality follows from Proposition~\ref{prop:LipLowBndDist}.
Finally, putting $x=x^{ k+1}$, $\bar{x}=x_\star\in X_\star$ in   \eqref{eq:ProxBasic} and minimizing both sides we obtain
\begin{equation}\label{eq:LowStrConvex}
F(x^{ k+1})-F_\star\geq\tfrac{\mu_f}{2}\|x^{ k+1}-x_\star\|^2.
\end{equation}
Combining~\eqref{eq:distX1} and~\eqref{eq:LowStrConvex} we arrive at~\eqref{eq:strC2}.
\iftoggle{svver}{\qed}{\end{proof}}

\end{document}